\documentclass{article}
\usepackage{graphicx} 
\usepackage[english]{babel}
\usepackage{amsthm}
\usepackage{amsmath}
\usepackage{float}
\usepackage{mathtools}
\usepackage{amssymb}
\usepackage[sort]{natbib}
\usepackage[a4paper, left=3cm, right=3cm, top=2cm, bottom=2cm]{geometry}
\usepackage{hyperref}
\usepackage{comment}
\usepackage[noabbrev,capitalise]{cleveref}
\usepackage{algorithm}
\usepackage{algpseudocode}
\usepackage{tikz}
\usepackage{physics}
\usetikzlibrary{positioning, quotes, decorations.pathreplacing, shapes.geometric, calc,through,backgrounds}
\usepackage{tcolorbox}
\usepackage{lipsum} 
\usepackage[cmtip,all]{xy}
\usepackage{caption} 
\usepackage{wrapfig}
\usepackage{subcaption}
\usepackage{array}
\usepackage{tabularx}

\usepackage[T1]{fontenc}

\usepackage{todonotes}

\DeclareMathOperator{\Aut}{Aut}
\DeclareMathOperator{\Cay}{Cay}
\DeclareMathOperator{\modd}{modd}

\title{Polyhedral Maps of Cubic Graphs with given Automorphism Groups}

\author{
Ugo Detaille\thanks{RWTH Aachen University, Aachen, Germany. E-mail: {\tt ugo.detaille@rwth-aachen.de}.} \and 
Meike Weiß\thanks{ RWTH Aachen University, Aachen, Germany. E-mail: {\tt weiss@art.rwth-aachen.de}.}
\and
Reymond Akpanya\thanks{RWTH Aachen University, Aachen, Germany. E-mail: {\tt akpanya@art.rwth-aachen.de}.}
\thanks{The University of Sydney, Sydney, Australia. E-mail: {\tt reymond.akpanya@sydney.edu.au}.}
\and
Alice C.\ Niemeyer\thanks{RWTH Aachen University, Aachen, Germany. E-mail: {\tt alice.niemeyer@art.rwth-aachen.de}.}
}
\date{ }

\newtheorem{theorem}{Theorem}[section]
\newtheorem*{theorem*}{Theorem}
\newtheorem{lemma}[theorem]{Lemma}
\newtheorem{definition}[theorem]{Definition}

\newtheorem*{corollary*}{Corollary}
\newtheorem*{conjecture*}{Conjecture}

\newtheorem{remark}[theorem]{Remark}

\newtheorem*{question*}{Question}
\begin{document}

\maketitle


\begin{abstract} 
L.\ Babai introduced a method for constructing a cubic graph whose automorphism group is isomorphic to a given finite group $G$, obtained by modifying a corresponding Cayley graph of $G$.
Building on this approach, we construct a cubic graph that admits a polyhedral map whose automorphism group, as well as the automorphism group of the polyhedral map itself, is isomorphic to $G$.
\end{abstract}

\section{Introduction}
A common practice in group theory is to deepen the understanding of a finite group by considering it as the automorphism group of a combinatorial structure, see \cite{MR1211265,sabidussi,Cori} for instance. Any finite group $G=\langle S\rangle$ can be viewed as the automorphism group of the edge-coloured Cayley graph $\Cay_{G,S},$ where the edge colouring is induced by the set of generators $S$.
Since $\Cay_{G,S}$ is an edge-coloured directed graph, it is natural to ask whether $G$ can be realised as the automorphism group of an undirected graph without edge colours. This question has first been answered by R. Frucht in 1949. In \cite{frucht,FruchtHow, MR1557026}, Frucht introduced several constructions of undirected graphs having prescribed finite groups as automorphism groups. In \cite{SurfacesWithAuto}, an oversight has been noted in Frucht’s original cubic graph construction. This has led to a simplified construction of cubic graphs with prescribed automorphism groups for groups given by a 2-element generating set and a modified construction for groups given by more than two generators. Furthermore, the aforementioned work demonstrates that the resulting cubic graphs can be equipped with cycle double covers that are invariant under the natural action of their automorphism groups on the set of cycles.
We observe that two cycles of such a constructed cycle double cover can intersect in two or more edges. Hence, these constructed cycle double covers do not describe polyhedral maps.
Background on cycle double covers is given in \cite{seymour,szekeres}, while \cite{GraphsOnSurfaces,BrehmSchulte1997,TopologicalGraphTheory} offer introductions to the theory of (polyhedral) maps.
An alternative construction of a cubic graph with prescribed automorphism group has been introduced by L.~Babai. In \cite{babaiBook}, Babai exploits a Cayley graph of a given finite group $G$ given with a generating set in order to realise $G$ as the automorphism group of a cubic graph, see \Cref{secPreliminaries}. Note that these cubic graphs are 2- but not 3-connected.

Inspired by Babai's cubic graph construction in \cite{babaiBook} and the construction of cycle double covers in \cite{SurfacesWithAuto}, we investigate the construction of cubic graphs with prescribed automorphism groups and corresponding cycle double covers. In particular, we prove the following theorem.

\begin{theorem*}
    Let $G$ be a finite group and $S$ a generating set for $G$. Then there exist an $3$-connected cubic undirected graph $D_{G,S}$ and a cycle double cover $\mathcal{Z}$ of $D_{G,S}$ satisfying the following properties: 
    \begin{enumerate}
        \item[(1)] $\Aut(D_{G,S}) \cong G,$
        \item[(2)] $\mathcal{Z}$ describes a polyhedral map, and 
        \item[(3)]$ {}^{\Aut(D_{G,S})}\mathcal{Z}:=\{\pi(z)\mid z\in \mathcal{Z}, \pi\in\Aut(D_{G,S})\}=\mathcal{Z},$ i.e., $\mathcal{Z}$ is $\Aut(D_{G,S})$-invariant.
    \end{enumerate}
\end{theorem*}
This result will be established in \Cref{theoremAutGroupD,theoremZ1CDC,lemmaInvarianceCDC}.
By proving the above theorem, we extend the results of \cite{SurfacesWithAuto} by modifying the construction in \cite{babaiBook}. We further establish a relationship between the cycles of our constructed cycle double covers and the left cosets of the prescribed groups, see \Cref{lemmaCyclesPowers}.
Babai's construction seems fruitful for graph constructions with various given constraints. Currently, we are exploring whether Babai's cubic graph construction can be exploited to construct $k$-regular graphs with polyhedral embedding.

Our paper is structured as follows: In \cref{secPreliminaries}, we introduce basic notations and fundamentals relevant to our work. In addition, we briefly recall Babai’s earlier construction of cubic graphs from \cite{babaiBook}. In \cref{secGraphD}, we present our cubic graph construction and formally prove that the automorphism group of a constructed cubic graph is in fact isomorphic to a prescribed group. Next, we construct cycle double covers for our cubic graphs that induce polyhedral maps, see \cref{sec1CutCDC}. Finally, in \cref{secProperties}, we provide further insights on the connection between the group and our graph construction, as well as certain combinatorial properties of our polyhedral maps and their underlying cubic graphs. This paper is based on the bachelor thesis of the first author \cite{gitBA}. 
Additionally, \cite{gitBA} contains a GAP-implementation~\cite{Digraphs,GAP4} of a previous graph construction based on the aforementioned bachelor thesis. Furthermore, the same repository contains a Magma-implementation \cite{Bosma} of the graph construction. We have tested the implemented code with randomly selected groups of reasonable orders to ensure that our code and our results are correct and reliable.

\section{Preliminaries}\label{secPreliminaries}
For $k \in \mathbb{N}$, we write $[k]$ to denote the set $\{1, \ldots, k \} \subseteq \mathbb{N}$.
We define the modified $\bmod$-operator called $\modd$ as follows:
\[ \modd : \mathbb{N}^2 \longrightarrow \mathbb{N}, \; (n,m) \longmapsto 
     \begin{cases*}
        n, & \textnormal{if $m \vert n$}, \\
        n \mod m, & \textnormal{else.}
     \end{cases*}
    \]
Essentially, the operator $\modd$ simply adjusts the indices of the elements such that this fits to our enumerations.

Given a graph $\Gamma =(V,E)$, we write $\{ v,w\} \in E$ for an undirected edge between two vertices $v,w\in V$, and $(v,w)\in E$ when $\Gamma$ is directed.
We assume that all graphs in this paper are connected, simple and finite. 
Let $\Gamma = (V,E)$ be such a graph which is additionally undirected.
A connected 2-regular subgraph of $\Gamma$ is called a \textbf{cycle}. A \textbf{cycle double cover (CDC)} of $\Gamma$ is a set of cycles in $\Gamma$ such that every edge in $E$ is contained in exactly two of these cycles. A CDC of a cubic graph induces a $\textbf{polyhedral map}$ if any two cycles of the CDC intersect in at most one edge.
Note that a cubic graph with a CDC inducing a polyhedral map is 3-connected, see \cite{GraphsOnSurfaces} for more details.
An \textbf{automorphism} of a (directed) graph $\Gamma$ is a bijection $\pi: V(\Gamma) \to V(\Gamma)$ such that every pair of vertices $v_1,v_2$ is connected by a (directed) edge in $\Gamma$ if and only if the same holds for $\pi(v_1),\pi(v_2)$. Further, if $e=\{v_1,v_2\}$ is an edge and $\pi\in \Aut(\Gamma)$ an automorphism, we define $\pi(e):=\{\pi(v_1),\pi(v_2)\}$. We define the image of an edge similarly if $\Gamma$ is a directed graph.
If $\Gamma$ has an edge colouring $\mathcal{C}:E(\Gamma) \rightarrow \left\{ c_1, \ldots, c_k\right\}$, an automorphism $\pi$ of $\Gamma$ must also satisfy $\mathcal{C}(e)=\mathcal{C}(\pi(e))$ for every $e\in E(\Gamma)$. We denote the group of automorphisms of $\Gamma$ by $\Aut(\Gamma)$. Next, we introduce the Cayley graph of a group with a given generating set, as this notion is central to our investigation.
\begin{definition}\label{defCayleyGraph}
    Let $G$ be a finite group and $S = \{s_1, \ldots, s_k\}$ a generating set for $G$. The \textbf{\emph{Cayley graph}} $\Cay_{G,S}=(V,E)$ of $G$ with respect to $S$ is the directed and edge-coloured graph defined by
    \begin{enumerate}
        \item $V := \left\{g \mid g \in G \right\},$
        \item $E := \left\{ (g_1, g_2) \in G^2 \mid \textnormal{there exists } s \in S \textnormal{ such that } g_1 \cdot s = g_2 \right\}$ \text{ and}
        \item $e = (g_1, g_2) \in E$ \textnormal{has colour} $d \in [k] , \textnormal{ if } g_1 \cdot s_d = g_2$.
    \end{enumerate}
\end{definition}
For a finite group $G$ generated by a set $S= \{s_1, \ldots, s_k\}$, the Cayley graph $\Cay_{G,S}$ is strongly connected. Moreover, it is known that the automorphism group of the Cayley graph, respecting edge orientation and edge colouring, is given by the left-regular action of $G$ on itself, i.e.\ 
 \[  \Aut(\Cay_{G,S}) = \left\{ \pi_g \in \textnormal{Sym}(G) \mid \pi_g(h) := gh \textnormal{ for all } h \in G \right\} =: \mathcal{L}(G). \]
We conclude this section by recalling the cubic graph construction from \cite{babaiBook}. Starting with a finite group $G$ and a generating set $S$ for $G$, Babai's construction computes the Cayley graph $\Cay_{G,S}$ as a first step. As stated above, the automorphism group of this edge-coloured directed graph is isomorphic to $G$. Babai's cubic graph construction is based on the idea of modifying the given Cayley graph such that the resulting graph is undirected and cubic. Hence, the vertices and edges of the Cayley graph have to be modified in such a way that both the orientations and the colours of the directed edges are preserved. To achieve this, Babai replaces the vertices and edges of the Cayley graph $\Cay_{G,S}$ by introducing certain subgraphs. In this work, we refer to such subgraphs as \textbf{gadgets}. In particular, Babai replaces a directed edge $(g,h) \in E(\Cay_{G,S})$ of colour $d \in [k]$ as presented in \cref{fig:dPathGToH}. The colour $d$ is preserved by using $d$ many small gadgets glued to one another, and the orientation is always encoded by connecting the endpoint of the edge --- in our case $h$ --- to a large gadget.
Moreover, Babai replaces each vertex of the Cayley graph $\Cay_{G,S}$ with a cycle of length $2\vert S \vert $, connecting each of the resulting $2|S|$ vertices to exactly one of the newly introduced gadgets that replace the original Cayley edges, see \cref{fig:Dummy1}.

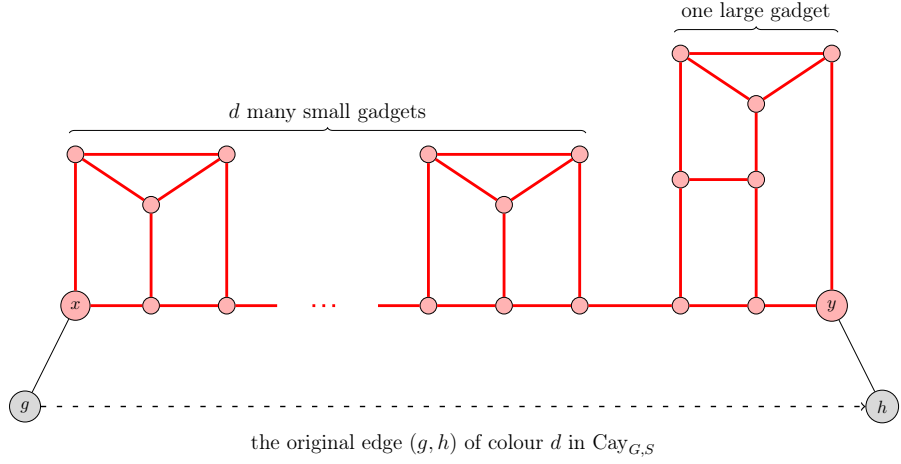
\begin{figure}[ht]
    \centering
    \resizebox{12cm}{!}{
    \begin{tikzpicture}[
        node distance = 5mm and 5mm,
             V/.style = {circle, draw, fill=gray!30},
        every edge quotes/.style = {auto, font=\footnotesize, sloped}
                            ]
    
    \tikzstyle{overbrace text style}=[font=\large, above, pos=.5, yshift=4mm]
    \tikzstyle{underbrace text style}=[font=\large, below, pos=.5, yshift=-4mm]
    \tikzstyle{overbrace style}=[decorate,decoration={brace,raise=3mm,amplitude=3pt}]
    \tikzstyle{underbrace style}=[decorate,decoration={brace,raise=3mm,amplitude=3pt,mirror}]
    
    \begin{scope}[nodes=V]
    \node[] (g) at (-1,-5)   {$g$};
       
    \node[fill=red!30] (1) at (0,0)   {};
    \node[fill=red!30] (2) at (3,0)   {};
    \node[fill=red!30] (3) at (1.5, -1)          {};
    \node[fill=red!30] (4) at (0, -3)    {$x$};
    \node[fill=red!30] (5) at (1.5, -3)    {};
    \node[fill=red!30] (6) at (3, -3)          {};
    
    \node[fill=red!30] (7) at (7,0)   {};
    \node[fill=red!30] (8) at (10,0)   {};
    \node[fill=red!30] (9) at (8.5, -1)          {};
    \node[fill=red!30] (10) at (7, -3)    {};
    \node[fill=red!30] (11) at (8.5, -3)    {};
    \node[fill=red!30] (12) at (10, -3)          {};
    
    \node[fill=red!30] (13) at (12,2)   {};
    \node[fill=red!30] (14) at (15,2)   {};
    \node[fill=red!30] (15) at (13.5, 1)          {};
    \node[fill=red!30] (16) at (12, -0.5)    {};
    \node[fill=red!30] (17) at (13.5, -0.5)    {};
    \node[fill=red!30] (18) at (12, -3)          {};
    \node[fill=red!30] (19) at (13.5, -3)          {};
    \node[fill=red!30] (20) at (15, -3)          {$y$};
    
    \node[] (h) at (16, -5) {$h$};
            \end{scope}
        \draw   
                (g) edge[] (4)
                (1)  edge[ultra thick, red] (2)
                (1)  edge[ultra thick, red] (3)
                (2)  edge[ultra thick, red] (3)
                (1) edge[ultra thick, red] (4)
                (3) edge[ultra thick, red] (5)
                (2) edge[ultra thick, red] (6)
                (4) edge[ultra thick, red] (5)
                (5) edge[ultra thick, red] (6)
    
                (6) edge [ultra thick, red] (4,-3)
    
                (4.7, -3) edge[loosely dotted, ultra thick, red] (5.2, -3)
    
                (6,-3) edge[ultra thick, red] (10)
    
                (7)  edge[ultra thick, red] (8)
                (7)  edge[ultra thick, red] (9)
                (8)  edge[ultra thick, red] (9)
                (7) edge[ultra thick, red] (10)
                (9) edge[ultra thick, red] (11)
                (8) edge[ultra thick, red] (12)
                (10) edge[ultra thick, red] (11)
                (11) edge[ultra thick, red] (12)
    
                (12) edge[ultra thick, red] (18)
    
            (13)  edge[ultra thick, red] (14)
            (13)  edge[ultra thick, red] (15)
            (14)  edge[ultra thick, red] (15)
            (13) edge[ultra thick, red] (16)
            (15) edge[ultra thick, red] (17)
            (16) edge[ultra thick, red] (17)
            (16) edge[ultra thick, red] (18)
            (17) edge[ultra thick, red] (19)
            (14) edge[ultra thick, red] (20)
            (18) edge[ultra thick, red] (19)
            (19) edge[ultra thick, red] (20)
    
            (20) edge[] (h)
    
            (g) edge[loosely dashed, thick, ->] node[underbrace text style] {the original edge $(g,h)$ of colour $d$ in $\textnormal{Cay}_{G,S}$} (h);
    
    \draw [overbrace style] (1.north west) -- (8.north east) node [overbrace text style] {$d$ many small gadgets};
    
    \draw [overbrace style] (13.north west) -- (14.north east) node [overbrace text style] {one large gadget};
    
\end{tikzpicture}
    }
    \caption{ \centering The subgraph replacing the original Cayley graph edge $(g,h)$ of colour $d \in [k]$.}
    \label{fig:dPathGToH}
\end{figure} 

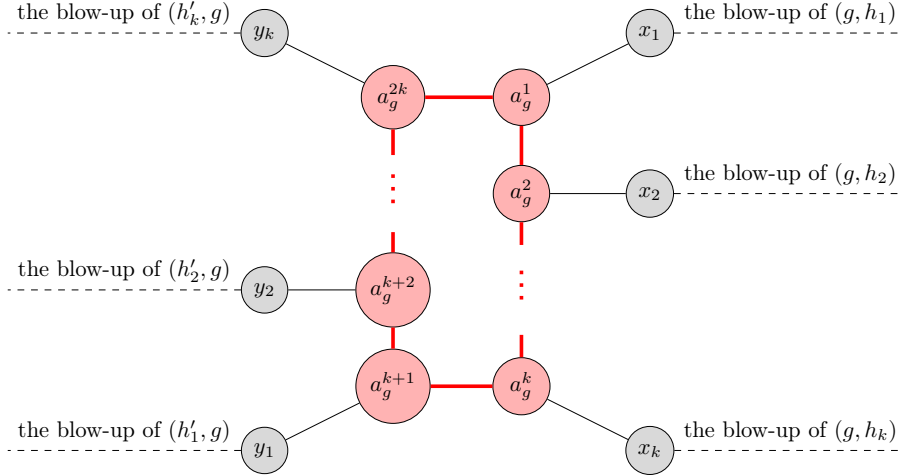
\begin{figure}[ht]
    \begin{center}
    \resizebox{12cm}{!}{\begin{tikzpicture}[
    node distance = 2mm and 2mm,
         V/.style = {circle, draw, fill=gray!30},
    every edge quotes/.style = {auto, font=\footnotesize, sloped}
                        ]

\tikzstyle{overbrace text style}=[font=\large, above, pos=.5, yshift=4mm]
\tikzstyle{underbrace text style}=[font=\large, below, pos=.5]
\tikzstyle{overbrace style}=[decorate,decoration={brace,raise=3mm,amplitude=3pt}]

\begin{scope}[nodes=V]
\node[fill=red!30] (1) at (0,0) {$a_g^1$};
\node[fill=red!30] (2) at (0,-1.5) {$a_g^2$};
\node[fill=red!30] (3) at (0,-4.5) {$a_g^k$};
\node[fill=red!30] (4) at (-2, -4.5) {$a_g^{k+1}$};
\node[fill=red!30] (5) at (-2, -3) {$a_g^{k+2}$};
\node[fill=red!30] (6) at (-2, 0) {$a_g^{2k}$};

\node (7) at (2, 1) {$x_1$};
\node (8) at (2, -1.5) {$x_2$};
\node (9) at (2, -5.5) {$x_k$};
\node (10) at (-4, -5.5) {$y_1$};
\node (11) at (-4, -3) {$y_2$};
\node (12) at (-4, 1) {$y_k$};

\end{scope}

\draw   
      (1) edge[red, ultra thick] (6)
      (1) edge[red, ultra thick] (2)
      (3) edge[red, ultra thick] (4)
      (4) edge[red, ultra thick] (5)
      
      (2) edge[red, ultra thick] (0, -2.3)
      (0, -2.7) edge[loosely dotted, ultra thick, red] (0, -3.2)
      (3) edge[red, ultra thick] (0, -3.7)

      (6) edge[red, ultra thick] (-2, -0.9)
      (-2, -1.2) edge[loosely dotted, ultra thick, red] (-2, -1.7)
      (5) edge[red, ultra thick] (-2, -2.1)

      (1) edge (7)
      (2) edge (8)
      (3) edge (9)
      (4) edge (10)
      (5) edge (11)
      (6) edge (12)

      (7) edge [dashed] node[black, sloped, above] {the blow-up of $(g, h_1)$} (6, 1)
      (8) edge [dashed] node[black, sloped, above] {the blow-up of $(g, h_2)$} (6, -1.5)
      (9) edge [dashed] node[black, sloped, above] {the blow-up of $(g, h_k)$} (6, -5.5)
      (10) edge [dashed] node[black, sloped, above] {the blow-up of $(h_1^{\prime}, g)$} (-8, -5.5)
      (11) edge [dashed] node[black, sloped, above] {the blow-up of $(h_2^{\prime}, g)$} (-8, -3)
      (12) edge [dashed] node[black, sloped, above] {the blow-up of $(h_k^{\prime}, g)$} (-8, 1);
        
\end{tikzpicture}}
    \end{center}
    \caption{\centering The blow-up of a vertex $g \in V(\Cay_{G,S})$ into a cycle of length $2k$ and its immediate neighbourhood.}
    \label{fig:Dummy1}
\end{figure}
\Cref{fig:dPathGToH} illustrates the fact that the cubic graphs resulting from Babai's construction are 2- but not 3-connected. In particular, for any finite group and any choice of generating set, removing the vertices $x$ and $y$ shown in \cref{fig:dPathGToH} results in a disconnected graph.

\section{Our Graph Construction \texorpdfstring{$D_{G,S}$}{D}}\label{secGraphD}
In this section, we present our cubic graph construction for a finite group $G$ and a generating set $S=\{s_1,\ldots,s_k\}$ for $G$. Similarly to Babai's approach, we construct subgraphs that serve to replace both the edges and the vertices of the Cayley graph $\Cay_{G,S}$.
In \cref{subSecGraphD}, we detail the construction of the cubic graph, and in \cref{subSecAutD}, we establish that its automorphism group is isomorphic to the input group $G$.

\subsection{Constructing the Graph $D_{G,S}$}\label{subSecGraphD}
First, we introduce a key element of our construction, a graph with ten vertices and thirteen edges, shown in \cref{fig:NewGadgets}, which we refer to as the \textbf{endgadget}.
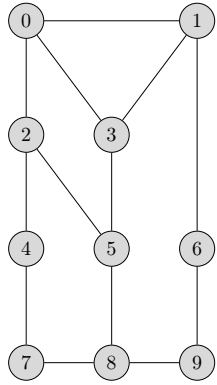
\begin{figure}[ht]
    \centering
    \resizebox{3cm}{!}{
            \begin{tikzpicture}[
node distance = 5mm and 5mm,
     V/.style = {circle, draw, fill=gray!30},
every edge quotes/.style = {auto, font=\footnotesize, sloped}
                    ]
    \begin{scope}[nodes=V]
\node (1) at (0,0)   {0};
\node (2) at (3,0)   {1};
\node (3) at (0, -2)          {2};
\node (4) at (1.5, -2) {3};
\node (5) at (0, -4) {4};
\node (6) at (1.5, -4) {5};
\node (7) at (3, -4) {6};
\node (8) at (0, -6) {7};
\node (9) at (1.5, -6) {8};
\node (10) at (3, -6) {9};

    \end{scope}
\draw   (1)  edge (2)
        (1)  edge (3)
        (1)  edge (4)
        (2) edge (4)
        (3) edge (5)
        (4) edge (6)
        (2) edge (7)
        (3) edge (6)
        (5) edge (8)
        (6) edge (9)
        (7) edge (10)
        (8) edge (9)
        (9) edge (10);
        
    \end{tikzpicture}
        }
\caption{The endgadget.}
\label{fig:NewGadgets}
\end{figure}

Next, we define the undirected subgraph that replaces a directed edge of the Cayley graph, while preserving both the colour and orientation of the edge in $\Cay_{G,S}$.
For an edge $(g,h) \in E(\Cay_{G,S})$ of colour $d \in [k]$ with respect $G=\langle S\rangle$, we define the \textbf{$d$-chain} from $g$ to $h$ as the graph $\mathfrak{C}_{[d,g,h]}$ consisting of a ladder of length $d$ combined with an endgadget, as depicted in \cref{fig:dPathD1}. Thus, the colour $d$ is encoded by the length of the ladder. Note, in the following we will introduce further gadgets that replace the nodes $g,h$ in the Cayley graph. For the moment, it suffices to reference these subgraphs as $g$ and $h.$ With this the orientation of $(g,h)$ is maintained by connecting $g$ to the ladder itself, while the end vertex $h$ is connected to the only endgadget of the chain.
We replace the edge $(g,h)$ by the $d$-chain and connect $g$ to the ladder and $h$ to the endgagdet as depicted in \cref{fig:dPathD1}.
\begin{figure}[ht]
    \centering
    \resizebox{12cm}{!}{
    
    \begin{tikzpicture}[
    node distance = 5mm and 5mm,
         V/.style = {circle, draw, fill=gray!30},
    every edge quotes/.style = {auto, font=\footnotesize, sloped}
                        ]

\tikzstyle{overbrace text style}=[font=\large, above, pos=.5, yshift=4mm]
\tikzstyle{underbrace text style}=[font=\large, below, pos=.5, yshift=-4mm]
\tikzstyle{overbrace style}=[decorate,decoration={brace,raise=3mm,amplitude=3pt}]
\tikzstyle{underbrace style}=[decorate,decoration={brace,raise=3mm,amplitude=3pt,mirror}]

\begin{scope}[nodes=V]

\node[fill=red!30] (a) at (1,-2)   {$u_1$};
\node[fill=red!30] (b) at (3,-2)   {$u_2$};
\node[fill=red!30] (c) at (1,-4)   {$\ell_1$};
\node[fill=red!30] (d) at (3,-4)   {$\ell_2$};

\node[fill=red!30] (e) at (9,-2)   {$u_d$};
\node[fill=red!30] (f) at (9,-4)   {$\ell_d$};

\node[fill=red!30] (0) at (11, 2) {0};
\node[fill=red!30] (1) at (14,2) {1};
\node[fill=red!30] (2) at (11,0)   {2};
\node[fill=red!30] (3) at (12.5, 0)          {3};
\node[fill=red!30] (4) at (11, -2)    {4};
\node[fill=red!30] (6) at (14, -2)    {6};
\node[fill=red!30] (7) at (11, -4)          {7};
\node[fill=red!30] (8) at (12.5, -4)          {8};
\node[fill=red!30] (9) at (14, -4)          {9};
\node[fill=red!30] (5) at (12.5, -2) {5};

        \end{scope}
    \draw   

            (a) edge[ultra thick, red] (b)
            (a) edge[ultra thick, red] (c)
            (c) edge[ultra thick, red] (d)
            (b) edge[ultra thick, red] (d)

            (e) edge[ultra thick, red] (4)
            (f) edge[ultra thick, red] (7)
            (e) edge[ultra thick, red] (f)

        (0) edge[ultra thick, red] (1)
        (0) edge[ultra thick, red] (3)
        (0) edge[ultra thick, red] (2)
        (1) edge[ultra thick, red] (3)
        (3) edge[ultra thick, red] (5)
        (2) edge[ultra thick, red] (4)
        (2) edge[ultra thick, red] (5)
        (4) edge[ultra thick, red] (7)
        (1) edge[ultra thick, red] (6)
        (5) edge[ultra thick, red] (8)
        (4) edge[ultra thick, red] (7)
        (7) edge[ultra thick, red] (8)
        (8) edge[ultra thick, red] (9)
        (6) edge[ultra thick, red] (9)
        

        (b) edge[ultra thick, red] (4.5, -2)
        (d) edge[ultra thick, red] (4.5, -4)

        (e) edge[ultra thick, red] (7.5, -2)
        (f) edge[ultra thick, red] (7.5, -4)

        (5.75, -2) edge[loosely dotted, ultra thick, red] (6.25, -2)
        (5.75, -4) edge[loosely dotted, ultra thick, red] (6.25, -4);


\draw [overbrace style] (a.north west) -- (e.north east) node [overbrace text style] {A ladder of length $d$};

\draw [overbrace style] (0.north west) -- (1.north east) node [overbrace text style] {One endgadget};



        \end{tikzpicture}
    }
    \caption{The chain $\mathfrak{C}_{[d,g,h]}$.}
    \label{fig:dPathD1}
\end{figure}
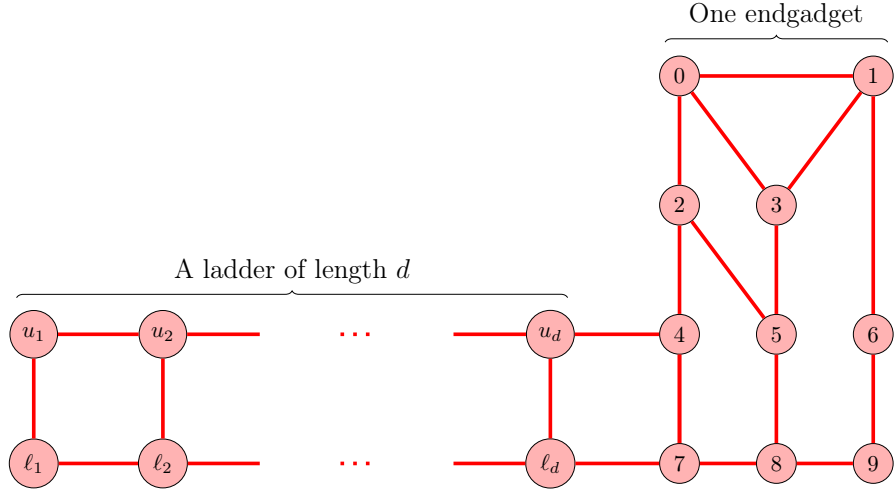
We write $\mathfrak{C}_{[d,g,h]}(v)$ to denote vertex $v$ of the $d$-chain connecting $g$ to $h$, where $v$ might be any label in the set $\left\{u_i, \ell_i, j \mid i \in [d], j \in \{0,\ldots, 9\}\right\}$. We always use the labels $u_i$ and $\ell_i$ for the vertices of the ladder, where $u$ stands for \emph{upper} and $\ell$ for \emph{lower}. Clearly, this notation is well-defined and unique for every vertex.
For $1\leq i <d,$ we refer to the edges of the form $\left\{ \mathfrak{C}_{[d,g,h]}(u_i), \mathfrak{C}_{[d,g,h]}(u_{i+1}) \right\}$ as \textbf{upper chain edges}, and to those of the form $\left\{ \mathfrak{C}_{[d,g,h]}(\ell_i), \mathfrak{C}_{[d,g,h]}(\ell_{i+1}) \right\}$ as the \textbf{lower chain edges}.

We build a second type of subgraph for every element $g \in G$, called the \textbf{blow-up graph of $g$}, denoted by $A_g$. For each element $g \in G$, we \emph{blow up} the corresponding vertex $g$ in $\Cay_{G,S}$ as depicted in \cref{fig:InteriorCycles}. In the following, we give a formal description of this modification.
We begin by describing the general case where $k = |S| \geq 3$. We will then adjust the constructed blow-up graph to the cases $k=1$ and $k=2$.
For every colour $d \in [k]$, we use four \textbf{connector vertices} $\{a_g^{d.1}, a_g^{d.2}, a_g^{d.3}, a_g^{d.4} \}$, and two \textbf{center vertices} $\{ c_g^{d.1} , c_g^{d.2}\}$ for $g\in G$. We refer to these vertices, together with the five edges connecting them, as the \textbf{$s_d$-block of $g$}, with reference to the generator $s_d$ associated with the colour $d$.
In particular, each $s_d$-block of $g$ contains the edges $\{a_g^{d.1},a_g^{d.2}\}$, $\{a_g^{d.2}, c_g^{d.1}\}$, $\{c_g^{d.1},
a_g^{d.3}\}$, $\{a_g^{d.3},a_g^{d.4}\}$ and $\{ c_g^{d.1} , c_g^{d.2}\}$.
We build an interior cycle over all center vertices with superscript 2, i.e.\ a cycle over all $c_g^{d.2}$ for $d \in [k]$. Finally, we connect all $s_d$-blocks to one another by adding the edge $\left\{ a_g^{d.4} , a_g^{(d+1 \modd k).1} \right\}$ between each pair of blocks for all $d \in [k]$. The resulting graph is the \textbf{blow-up graph $A_g$ of $g$}.
As mentioned, the cases $k = 1,2$ differ slightly from the above definition. When $k=1$, the center vertices are omitted entirely. For $k=2$, we use only one center vertex per colour $d \in [2]$. All three cases are depicted in \cref{fig:InteriorCycles}, where each $s_d$-block is highlighted by a distinct colour for its vertices and edges.

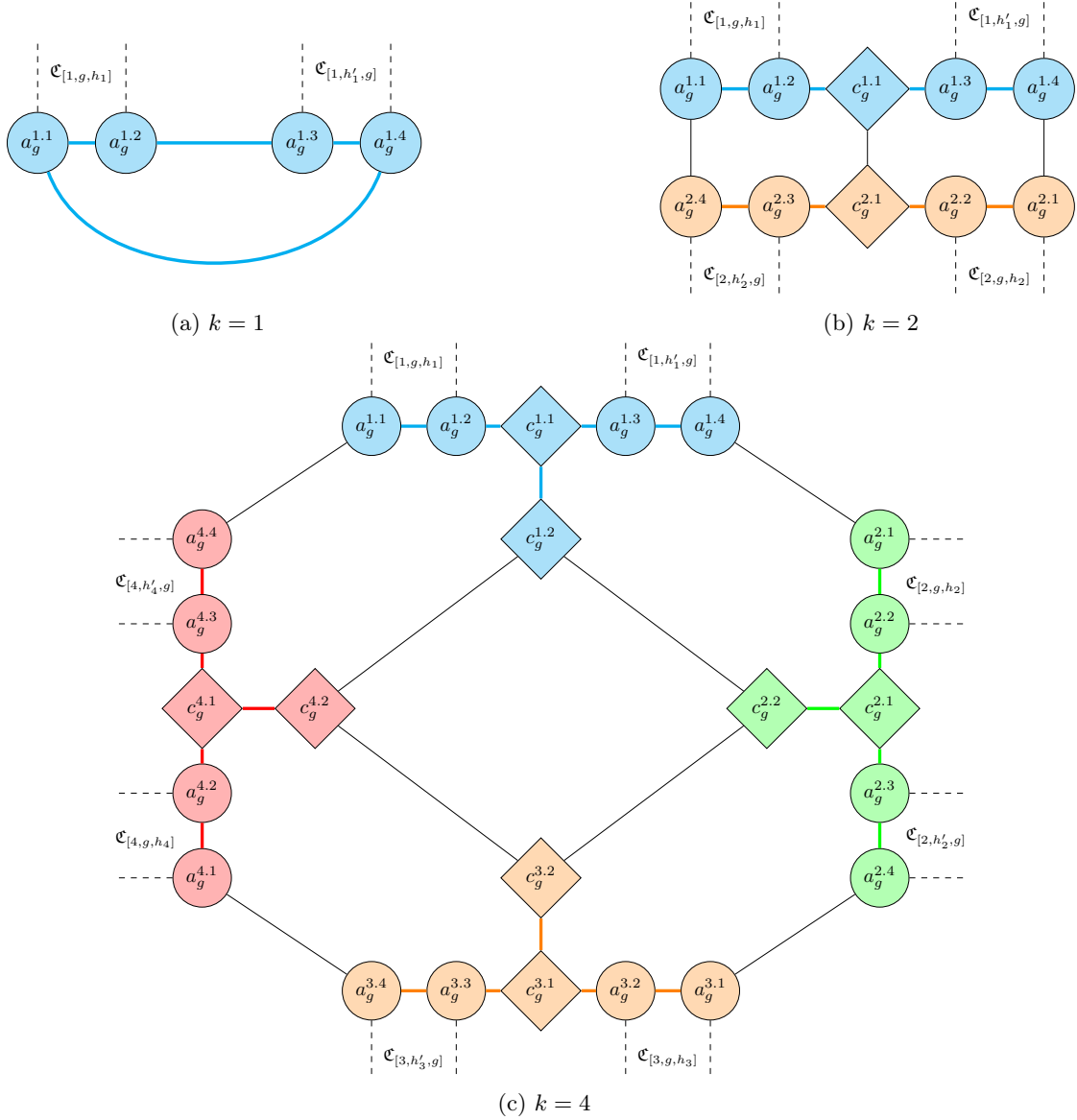
\begin{figure}[!ht]
    \begin{subfigure}[b]{0.4\textwidth}
    \centering
    \resizebox{\linewidth}{!}{
    \begin{tikzpicture}[
    node distance = 3mm and 3mm,
         V/.style = {circle, draw, fill=gray!30},
    every edge quotes/.style = {auto, font=\footnotesize, sloped}
                        ]

\tikzstyle{overbrace text style}=[font=\large, above, pos=.5, yshift=4mm]
\tikzstyle{underbrace text style}=[font=\large, below, pos=.5, yshift=4mm]
\tikzstyle{overbrace style}=[decorate,decoration={brace,raise=3mm,amplitude=3pt}]
\tikzstyle{underbrace style}=[decorate,decoration={brace,raise=3mm,amplitude=3pt,mirror}]

\begin{scope}[nodes=V]

\node[circle, draw, fill=cyan!30] (7) at (3,1)   {$a_g^{1.1}$};
\node[circle, draw, fill=cyan!30] (8) at (4.5,1)   {$a_g^{1.2}$};
\node[circle, draw, fill=cyan!30] (11) at (7.5,1)   {$a_g^{1.3}$};
\node[circle, draw, fill=cyan!30] (12) at (9,1)   {$a_g^{1.4}$};

\end{scope}

\draw

        (7) edge[dashed] (3, 2.7)
        (8) edge[dashed] (4.5, 2.7)
        (11) edge[dashed] (7.5, 2.7)
        (12) edge[dashed] (9, 2.7)

        (7) edge[ultra thick, cyan] (8)
        (8) edge[ultra thick, cyan] (11)
        (11) edge[ultra thick, cyan] (12)
        (7) edge[bend right = 70, ultra thick, cyan] (12);

\draw [draw = none] (3, 1.5) -- (4.5, 1.5) node [overbrace text style, font=\small] {$\mathfrak{C}_{[1,g,h_1]}$};

\draw [draw = none] (7.5, 1.5) -- (9, 1.5) node [overbrace text style, font=\small] {$\mathfrak{C}_{[1,h_1^{\prime}, g]}$};

        \end{tikzpicture}
    }
    \caption{$k=1$}
    \label{fig:dummyNode1Gen}
    \end{subfigure}
    \hfill
    \begin{subfigure}[b]{0.4\textwidth}
        \centering
        \resizebox{\linewidth}{!}{
        \begin{tikzpicture}[
    node distance = 3mm and 3mm,
         V/.style = {circle, draw, fill=gray!30},
    every edge quotes/.style = {auto, font=\footnotesize, sloped}
                        ]

\tikzstyle{overbrace text style}=[font=\large, above, pos=.5, yshift=4mm]
\tikzstyle{underbrace text style}=[font=\large, below, pos=.5, yshift=4mm]
\tikzstyle{overbrace style}=[decorate,decoration={brace,raise=3mm,amplitude=3pt}]
\tikzstyle{underbrace style}=[decorate,decoration={brace,raise=3mm,amplitude=3pt,mirror}]

\begin{scope}[nodes=V]

\node[circle, draw, fill=orange!30] (19) at (9,-1)   {$a_g^{2.1}$};
\node[circle, draw, fill=orange!30] (20) at (7.5,-1)   {$a_g^{2.2}$};
\node[shape = diamond, draw,fill=orange!30] (21) at (6,-1)   {$c_g^{2.1}$};
\node[circle, draw, fill=orange!30] (23) at (4.5,-1)   {$a_g^{2.3}$};
\node[circle, draw, fill=orange!30] (24) at (3,-1)   {$a_g^{2.4}$};

\node[circle, draw, fill=cyan!30] (7) at (3,1)   {$a_g^{1.1}$};
\node[circle, draw, fill=cyan!30] (8) at (4.5,1)   {$a_g^{1.2}$};
\node[shape = diamond, draw,fill=cyan!30] (9) at (6,1)   {$c_g^{1.1}$};
\node[circle, draw, fill=cyan!30] (11) at (7.5,1)   {$a_g^{1.3}$};
\node[circle, draw, fill=cyan!30] (12) at (9,1)   {$a_g^{1.4}$};

\end{scope}

\draw

        (7) edge[dashed] (3, 2.5)
        (8) edge[dashed] (4.5, 2.5)
        (11) edge[dashed] (7.5, 2.5)
        (12) edge[dashed] (9, 2.5)
        (9) edge (21)
        (12) edge (19)
        (24) edge (7)

        (24) edge[dashed] (3, -2.5)
        (23) edge[dashed] (4.5, -2.5)
        (20) edge[dashed] (7.5, -2.5)
        (19) edge[dashed] (9, -2.5)

        (7) edge[ultra thick, cyan] (8)
        (8) edge[ultra thick, cyan] (9)
        (9) edge[ultra thick, cyan] (11)
        (11) edge[ultra thick, cyan] (12)

        (19) edge[ultra thick, orange] (20)
        (20) edge[ultra thick, orange] (21)
        (21) edge[ultra thick, orange] (23)
        (23) edge[ultra thick, orange] (24);

\draw [draw = none] (3, 1.5) -- (4.5, 1.5) node [overbrace text style, font=\small] {$\mathfrak{C}_{[1,g,h_1]}$};
\draw [draw = none] (3, -2.3) -- (4.5, -2.3) node [underbrace text style, font=\small] {$\mathfrak{C}_{[2,h_2^{\prime}, g]}$};

\draw [draw = none] (7.5, 1.5) -- (9, 1.5) node [overbrace text style, font=\small] {$\mathfrak{C}_{[1,h_1^{\prime}, g]}$};
\draw [draw = none] (7.5, -2.3) -- (9, -2.3) node [underbrace text style, font=\small] {$\mathfrak{C}_{[2,g,h_2]}$};

        \end{tikzpicture}
        }
        \caption{$k=2$}
        \label{fig:interiorD2_2Gens}
    \end{subfigure}
    \begin{subfigure}[b]{\textwidth}
        \centering
        \resizebox{12cm}{!}{\begin{tikzpicture}[
    node distance = 3mm and 3mm,
         V/.style = {circle, draw, fill=gray!30},
    every edge quotes/.style = {auto, font=\footnotesize, sloped}
                        ]

\tikzstyle{overbrace text style}=[font=\large, above, pos=.5, yshift=4mm]
\tikzstyle{underbrace text style}=[font=\large, below, pos=.5, yshift=4mm]
\tikzstyle{overbrace style}=[decorate,decoration={brace,raise=3mm,amplitude=3pt}]
\tikzstyle{underbrace style}=[decorate,decoration={brace,raise=3mm,amplitude=3pt,mirror}]

\begin{scope}[nodes=V]

\node[circle, draw, fill=green!30] (13) at (12,0)   {$a_g^{2.1}$};
\node[circle, draw, fill=green!30] (14) at (12,-1.5)   {$a_g^{2.2}$};
\node[shape = diamond, draw,fill=green!30] (15) at (12,-3)   {$c_g^{2.1}$};
\node[shape = diamond, draw, fill=green!30] (16) at (10,-3)   {$c_g^{2.2}$};
\node[circle, draw, fill=green!30] (17) at (12, -4.5)   {$a_g^{2.3}$};
\node[circle, draw, fill=green!30] (18) at (12,-6)   {$a_g^{2.4}$};

\node[circle, draw, fill=orange!30] (19) at (9,-8)   {$a_g^{3.1}$};
\node[circle, draw, fill=orange!30] (20) at (7.5,-8)   {$a_g^{3.2}$};
\node[shape = diamond, draw,fill=orange!30] (21) at (6,-8)   {$c_g^{3.1}$};
\node[shape = diamond, draw, fill=orange!30] (22) at (6,-6)   {$c_g^{3.2}$};
\node[circle, draw, fill=orange!30] (23) at (4.5,-8)   {$a_g^{3.3}$};
\node[circle, draw, fill=orange!30] (24) at (3,-8)   {$a_g^{3.4}$};

\node[circle, draw, fill=red!30] (1) at (0,0)   {$a_g^{4.4}$};
\node[circle, draw, fill=red!30] (2) at (0,-1.5)   {$a_g^{4.3}$};
\node[shape = diamond, draw,fill=red!30] (3) at (0,-3)   {$c_g^{4.1}$};
\node[shape = diamond, draw, fill=red!30] (4) at (2,-3)   {$c_g^{4.2}$};
\node[circle, draw, fill=red!30] (5) at (0,-4.5)   {$a_g^{4.2}$};
\node[circle, draw, fill=red!30] (6) at (0,-6)   {$a_g^{4.1}$};

\node[circle, draw, fill=cyan!30] (7) at (3,2)   {$a_g^{1.1}$};
\node[circle, draw, fill=cyan!30] (8) at (4.5,2)   {$a_g^{1.2}$};
\node[shape = diamond, draw,fill=cyan!30] (9) at (6,2)   {$c_g^{1.1}$};
\node[shape = diamond, draw, fill=cyan!30] (10) at (6,0)   {$c_g^{1.2}$};
\node[circle, draw, fill=cyan!30] (11) at (7.5,2)   {$a_g^{1.3}$};
\node[circle, draw, fill=cyan!30] (12) at (9,2)   {$a_g^{1.4}$};

\end{scope}

\draw

        (7) edge[dashed] (3, 3.5)
        (8) edge[dashed] (4.5, 3.5)
        (11) edge[dashed] (7.5, 3.5)
        (12) edge[dashed] (9, 3.5)

        (24) edge[dashed] (3, -9.5)
        (23) edge[dashed] (4.5, -9.5)
        (20) edge[dashed] (7.5, -9.5)
        (19) edge[dashed] (9, -9.5)

        (1) edge[dashed] (-1.5, 0)
        (2) edge[dashed] (-1.5, -1.5)
        (5) edge[dashed] (-1.5, -4.5)
        (6) edge[dashed] (-1.5, -6) 

        (13) edge[dashed] (13.5, 0)
        (14) edge[dashed] (13.5, -1.5)
        (17) edge[dashed] (13.5, -4.5)
        (18) edge[dashed] (13.5, -6) 

        (1) edge[ultra thick, red] (2)
        (2) edge[ultra thick, red] (3)
        (3) edge[ultra thick, red] (4)
        (3) edge[ultra thick, red] (5)
        (5) edge[ultra thick, red] (6)

        (4) edge[] (10)
        (10) edge[] (16)
        (16) edge[] (22)
        (22) edge[] (4)

        (1) edge[] (7)

        (12) edge[] (13)
        (18) edge[] (19)
        (24) edge[] (6)

        (7) edge[ultra thick, cyan] (8)
        (8) edge[ultra thick, cyan] (9)
        (9) edge[ultra thick, cyan] (10)
        (9) edge[ultra thick, cyan] (11)
        (11) edge[ultra thick, cyan] (12)

        (13) edge[ultra thick, green] (14)
        (14) edge[ultra thick, green] (15)
        (15) edge[ultra thick, green] (16)
        (15) edge[ultra thick, green] (17)
        (17) edge[ultra thick, green] (18)

        (19) edge[ultra thick, orange] (20)
        (20) edge[ultra thick, orange] (21)
        (21) edge[ultra thick, orange] (22)
        (21) edge[ultra thick, orange] (23)
        (23) edge[ultra thick, orange] (24);

\draw [draw = none] (3, 2.5) -- (4.5, 2.5) node [overbrace text style, font=\small] {$\mathfrak{C}_{[1,g,h_1]}$};
\draw [draw = none] (3, -9.3) -- (4.5, -9.3) node [underbrace text style, font=\small] {$\mathfrak{C}_{[3,h_3^{\prime}, g]}$};

\draw [draw = none] (7.5, 2.5) -- (9, 2.5) node [overbrace text style, font=\small] {$\mathfrak{C}_{[1,h_1^{\prime}, g]}$};
\draw [draw = none] (7.5, -9.3) -- (9, -9.3) node [underbrace text style, font=\small] {$\mathfrak{C}_{[3,g,h_3]}$};

\draw [draw = none] (-1.5, -6) -- (-0.5, -6) node [overbrace text style, font=\small] {$\mathfrak{C}_{[4,g,h_4]}$};
\draw [draw = none] (13.5, -6) -- (12.5, -6) node [overbrace text style, font=\small] {$\mathfrak{C}_{[2,h_2^{\prime}, g]}$};

\draw [draw = none] (-1.5, -1.5) -- (-0.5, -1.5) node [overbrace text style, font=\small] {$\mathfrak{C}_{[4,h_4^{\prime}, g]}$};
\draw [draw = none] (13.5, -1.5) -- (12.5, -1.5) node [overbrace text style, font=\small] {$\mathfrak{C}_{[2,g,h_2]}$};

        \end{tikzpicture}}
        \caption{$k=4$}
        \label{fig:interiorD2}
    \end{subfigure}
    \caption{The blow-up graph $A_g$ for various $k$.}
    \label{fig:InteriorCycles}
\end{figure}

In the final construction step of our cubic graph, we connect the blow-up graphs to the $d$-chains.
Given an edge $(g,h)$ of colour $d$ in $E(\Cay_{G,S})$, we connect the vertex $a_g^{d.1}$ to the vertex $\mathfrak{C}_{[d,g,h]}(u_1)$, and $a_g^{d.2}$ to $\mathfrak{C}_{[d,g,h]}(\ell_1)$. Similarly, for an edge $(h^{\prime},g)$ of colour $d$ in the Cayley graph, we connect $a_g^{d.4}$ to $\mathfrak{C}_{[d,h^{\prime},g]}(6)$, and $a_g^{d.3}$ to $\mathfrak{C}_{[d,h^{\prime}, g]}(9)$. We illustrate this construction in \cref{fig:sdBlock}.
For completeness, we give a formal definition of our proposed cubic graph.

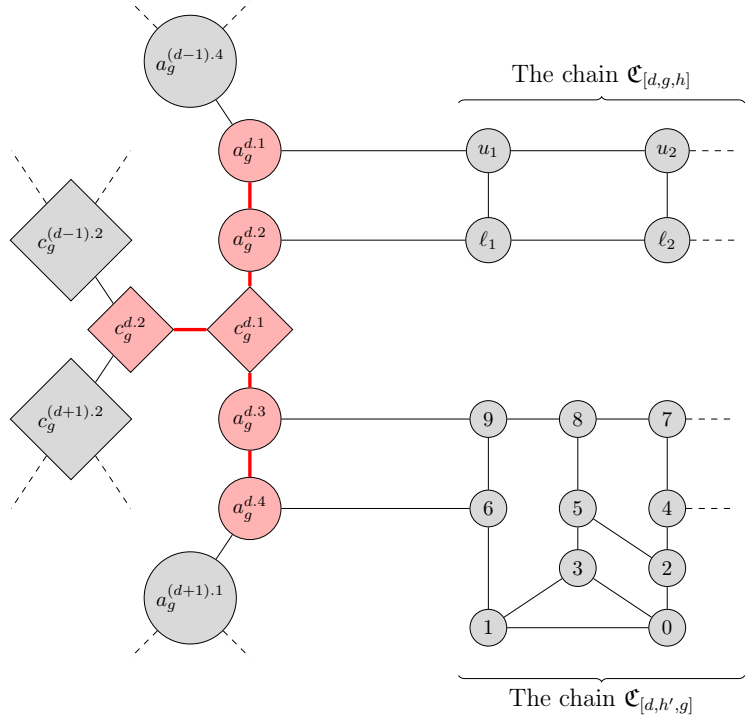
\begin{figure}[ht!]
    \centering
    \resizebox{10cm}{!}{
    \begin{tikzpicture}[
    node distance = 3mm and 3mm,
         V/.style = {circle, draw, fill=gray!30},
    every edge quotes/.style = {auto, font=\footnotesize, sloped}
                        ]

\tikzstyle{overbrace text style}=[font=\large, above, pos=.5, yshift=4mm]
\tikzstyle{underbrace text style}=[font=\large, below, pos=.5, yshift=-4mm]
\tikzstyle{overbrace style}=[decorate,decoration={brace,raise=3mm,amplitude=3pt}]
\tikzstyle{underbrace style}=[decorate,decoration={brace,raise=3mm,amplitude=3pt,mirror}]

\begin{scope}[nodes=V]

\node (a) at (4, 0)    {$u_1$};
\node (b) at (4, -1.5)          {$\ell_1$};

\node (c) at (7, 0) {$u_2$};
\node (d) at (7, -1.5) {$\ell_2$};

\node (a8) at (4, -4.5) {9};
\node (a5) at (4, -6) {6};
\node (a7) at (5.5, -4.5) {8};
\node (a6) at (7, -4.5) {7};
\node (a4) at (7, -6) {4};
\node (a3) at (5.5, -7) {3};
\node (a2) at (4, -8) {1};
\node (a1) at (7, -8) {$0$};
\node (a0) at (7, -7) {$2$};
\node (a9) at (5.5, -6) {$5$};

\node[circle, draw, fill=red!30] (1) at (0,0)   {$a_g^{d.1}$};
\node[circle, draw, fill=red!30] (2) at (0,-1.5)   {$a_g^{d.2}$};
\node[shape = diamond, draw,fill=red!30] (3) at (0,-3)   {$c_g^{d.1}$};
\node[shape = diamond, draw, fill=red!30] (4) at (-2,-3)   {$c_g^{d.2}$};
\node[circle, draw, fill=red!30] (5) at (0,-4.5)   {$a_g^{d.3}$};
\node[circle, draw, fill=red!30] (6) at (0,-6)   {$a_g^{d.4}$};

\node(27) at (-1, 1.5) {$a_g^{(d-1).4}$};
\node (28) at (-1, -7.5) {$a_g^{(d+1).1}$};

\node[shape = diamond] (25) at (-3, -1.5) {$c_g^{(d-1).2}$};
\node[shape = diamond] (26) at (-3, -4.5) {$c_g^{(d+1).2}$};

\end{scope}

\draw   

        (1) edge (a)
        (2) edge (b)

        (a) edge (c)
        (b) edge (d)
        (c) edge (d)
        (a) edge (b)

        (c) edge[dashed] (8.2,0)
        (d) edge[dashed] (8.2,-1.5)

        (a4) edge[dashed] (8.1, -6)
        (a6) edge[dashed] (8.1, -4.5)

        (5) edge (a8)
        (6) edge (a5)
        
        (27) edge[dashed] (-2, 2.5)
        (27) edge[dashed] (0, 2.5)

        (28) edge[dashed] (-2, -8.5)
        (28) edge[dashed] (0, -8.5)

        (25) edge[dashed] (-4,0)
        (25) edge[dashed] (-2, 0)

        (26) edge[dashed] (-4, -6)
        (26) edge[dashed] (-2, -6)

        (1) edge [] (27)
        (6) edge[] (28)

        (4) edge[] (25)
        (4) edge[] (26)

        (1) edge[ultra thick, red] (2)
        (2) edge[ultra thick, red] (3)
        (3) edge[ultra thick, red] (4)
        (3) edge[ultra thick, red] (5)
        (5) edge[ultra thick, red] (6)

        (a1) edge (a2)
        (a1) edge (a3)
        (a1) edge (a0)
        (a0) edge (a4)
        (a2) edge (a5) 
        (a2) edge (a3)
        (a3) edge (a9)
        (a7) edge (a9)     
        (a4) edge (a6)
        (a5) edge (a8)
        (a7) edge (a6)
        (a7) edge (a8)
        (a0) edge (a9);

\draw [overbrace style] (3.5, 0.5) -- (8.3, 0.5) node [overbrace text style] {The chain $\mathfrak{C}_{[d,g,h]}$};

\draw [underbrace style] (3.5, -8.5) -- (8.3, -8.5) node [underbrace text style] {The chain $\mathfrak{C}_{[d,h^{\prime},g]}$};

\end{tikzpicture}
    }
    \caption{An $s_d$-block (red) in $g$ and its immediate neighbourhood for $k \geq 3$.}
    \label{fig:sdBlock}
\end{figure}

\begin{definition}\label{defGraphD}
    Let $G$ be a finite group, $S$ a generating set for $G$ of cardinality $k \geq 1$, and $\mathcal{C}$ the $k$-edge-colouring of $\Cay_{G,S}$.
    We define the \textbf{\emph{graph $D_{G,S}$}} as $D_{G, S} := ( V_1 \cup V_2, E_1 \cup E_2 \cup E_3 \cup E_4)$, where
    \begin{enumerate}
        \item $V_1:= \bigcup_{g \in G} V(A_g)$ (the vertices of the blow-up graphs for all $g \in G$),
        \item $V_2 := \bigcup_{d \in [k]} \bigcup_{\substack{(g,h) \in E(\Cay_{G,S}) \\ \mathcal{C}((g,h)) = d}} V(\mathfrak{C}_{[d,g,h]})$ (the $d$-chain vertices for all edges $(g,h) \in E(\Cay_{G,S})$),
        \item $E_1 := \bigcup_{g \in G} E(A_g)$ (the edges of the blow-up graphs for all $g \in G$),
        \item $E_2 := \bigcup_{d \in [k]} \bigcup_{\substack{(g,h) \in E(\Cay_{G,S}) \\ \mathcal{C}((g,h)) = d}} E(\mathfrak{C}_{[d,g,h]})$ (the $d$-chain edges for all edges $(g,h) \in E(\Cay_{G,S})$),
        \item $E_3 := \bigcup_{d \in [k]} \bigcup_{\substack{(g,h) \in E(\Cay_{G,S}) \\ \mathcal{C}((g,h)) = d}} \left\{ \left\{ a_g^{d.1} , \mathfrak{C}_{[d,g,h]}(u_1) \right\} , \left\{ a_g^{d.2} , \mathfrak{C}_{[d,g,h]}(\ell_1) \right\} \right\}$ (the edges connecting a $d$-chain $\mathfrak{C}_{[d,g,h]}$ to the blow-up graph $A_g$).
        \item $E_4 := \bigcup_{d \in [k]} \bigcup_{\substack{(g,h) \in E(\Cay_{G,S}) \\ \mathcal{C}((g,h)) = d}} \left\{ \left\{ \mathfrak{C}_{[d,g,h]}(6) , a_h^{d.4} \right\} , \left\{ \mathfrak{C}_{[d,g,h]}(9) , a_h^{d.3} \right\} \right\}$ (the edges connecting a $d$-chain $\mathfrak{C}_{[d,g,h]}$ to the blow-up graph $A_h$).
    \end{enumerate}
\end{definition}

Note that the constructed graph $D_{G,S}$ is cubic for any finite group $G$ and generating set $S$.
Let $G$ be a group of order $n$ with a generating set $S$ of cardinality $k$. Then, for $k\geq 3$, the graph $D_{G,S}$ contains precisely $k^2n + 17kn$ vertices and $\frac{3k^2n+51kn}{2}$ edges. It has $36n$ vertices and $54n$ edges when $k=2$, and $16n$ vertices and $24n$ edges when $k=1$.
In general, the construction of the graph $D_{G,S}$
depends on the choice of the generating set $S$ for $G$ and the ordering of the generators in $S$. Consequently, changing the ordering of the generators in the above construction may result in non-isomorphic graphs. This can be verified in GAP using our implementations by constructing the graph $D_{G,S}$ for the group 
$G=A_5$ generated by $S=\{(1,5)(2,4), (1,2,4,3,5), (2,5,3) \}$. Here, we choose the orderings $(1,5)(2,4) < (1,2,4,3,5) < (2,5,3)$ and $(1,5)(2,4) < (2,5,3) < (1,2,4,3,5)$ and observe that the two resulting graphs are indeed not isomorphic. Note that the same example has been used in \cite{SurfacesWithAuto} to show that Frucht's cubic graph construction depends on the ordering of the generators of a given group.
\subsection{The Automorphism Group of the Graph $D_{G,S}$}\label{subSecAutD}
In this section, we prove that the automorphism group of the graph $D_{G,S}$ from \cref{defGraphD} is isomorphic to our initial group $G$. To do so, we first show that every automorphism of the edge-coloured Cayley graph $\Cay_{G,S}$ induces an automorphism of the graph $D_{G,S}$. Then, we prove that every automorphism of the graph $D_{G,S}$ already stems from the left multiplication with an element $g \in G$ within the group itself. 
To avoid redundancy, we only state the results for the cases in which the generating set $S$ of the given group $G$ has at least $k \geq 3$ elements. The remaining two cases $k \in [2]$ simply follow by omitting the images of the non-existing center vertices in the smaller blow-up graphs.
We start with the translation of automorphisms of the Cayley graph $\Cay_{G,S}$ to automorphisms of the graph $D_{G,S}$.

\begin{lemma}\label{lemmaAutCayleyAutD}
    Every edge-colour-respecting automorphism $\pi \in \Aut(\Cay_{G,S})$ of the edge-coloured graph $\Cay_{G,S}$ induces an automorphism $\tilde{\pi}$ of the graph $D_{G,S}$.
\end{lemma}
\begin{proof}
    First, we recall that $V(\Cay_{G,S}) = \{ g \mid g \in G\}$. Thus, let $g,h$ be elements in $G$ such that $gs_d = h$ for $s_d \in S$. Now, we define 
    \begin{equation*}
        \tilde{\pi} : V(D_{G,S}) \longrightarrow V(D_{G,S}), \; v \longmapsto \begin{cases}
            a_{\pi(g)}^{d.i} & \textnormal{if } v = a_g^{d.i} \textnormal{ for } d \in [k] \textnormal{ and } i \in [4], \\
            c_{\pi(g)}^{d.i} & \textnormal{if } v = c_g^{d.i} \textnormal{ for } d \in [k] \textnormal{ and } i \in [2], \\
            \mathfrak{C}_{[d,\pi(g),\pi(h)]}(v) & \textnormal{if } v = \mathfrak{C}_{[d,g,h]}(v) \textnormal{ for } d \in [k] .
        \end{cases}
    \end{equation*}
Clearly, $\tilde{\pi}$ is a well-defined bijection, and therefore $\tilde{\pi} \in \textnormal{Sym}(V(D_{G,S}))$. Moreover, it is easy to see that in fact all edge-relations of the graph are preserved. Thus, $\tilde{\pi} \in \Aut(D_{G,S}).$
\end{proof}

The remainder of this section is devoted to showing that the inverse direction holds. In particular, we show that every automorphism of the graph $D_{G,S}$ corresponds to an automorphism $\tilde{\pi}$ as constructed in the above proof.

\begin{lemma}\label{lemmaAutDAutCayley}
    Every automorphism $\pi$ of $D_{G,S}$ is induced by a left-regular action $\pi_g \in \mathcal{L}(G)$ of $G$.
\end{lemma}

We show the above lemma using partial results\, namely \Cref{lemmaDChainsPermus,lemmaDNGPermus,lemmaAutDLeftAction}. We start by analysing the possibilities of how an automorphism of $D_{G,S}$ can permute the subgraphs forming $d$-chains. For this purpose, we recycle the concept of a vertex \textbf{type}, a triple of numbers $(\lambda_1, \lambda_2, \lambda_3)$ introduced by Frucht in his study of graphs with a prescribed group \cite{frucht}. In a bridgeless cubic graph, the type of a vertex is given by a triple consisting of the lengths of shortest cycles through each pair of incident edges. Moreover, these lengths are arranged in ascending order.
Note that vertex types are a separating invariant under the automorphism group of a graph.
In \Cref{table:types}, we provide the vertex types of the ten vertices in an arbitrary endgadget of $D_{G,S}$. The values can be easily determined by reviewing \cref{fig:dPathD1,fig:InteriorCycles,fig:sdBlock}.

\begin{table}[ht]
    \centering
        \begin{tabular}{ | >{\centering}p{1.5cm}| >{\centering}p{0.5cm}| >{\centering}p{0.5cm}| >{\centering}p{0.5cm}|  }
            \hline
            \multicolumn{4}{|c|}{Types of the Endgadget} \tabularnewline [0.5ex] 
            \hline
            \centering Vertex $j$ & $\lambda_1^j$ & $\lambda_2^j$ & $\lambda_3^j$ \tabularnewline 
            \hline
            0 & 3    &4&  5 \tabularnewline
            1&   3  & 6  &7 \tabularnewline
            2&4 & 5&  7 \tabularnewline
            3 &3 & 4& 5 \tabularnewline
            4&   4  & 5& 7 \tabularnewline
            5& 4  & 5   &6 \tabularnewline
            6 & 4 & 6 & 8 \tabularnewline
            7 & 4 & 5 & 7 \tabularnewline
            8 & 5 & 6 & 8 \tabularnewline
            9 & 4 & 6 & 8 \tabularnewline
            \hline
           \end{tabular}
    \caption{The types for the vertices $\mathfrak{C}_{[d,g,h]}(j)$ of an endgadget in $D_{G,S}$ }
    \label{table:types}
\end{table}
    
By observing \cref{fig:InteriorCycles,fig:dPathD1}, it is easy to check that all connector vertices $a_g^{d.i}$ have the type $(4,8,10)$, and all ladder vertices the type $(4,4,6)$. Within our graph $D_{G,S}$, both types are unique to these two groups of vertices.
We denote the set of vertices of all blow-up graphs by $V(A_G)$, and the set of vertices of all $d$-chains by $V(\mathfrak{C}_G)$. Decomposing $V(D_{G,S})$ into $V(A_G)$ and $V(\mathfrak{C}_G)$ yields a partition of $V(D_{G,S})$.
Every vertex contained in $V(A_G)$ has at least two neighbours within $V(A_G)$, whereas every vertex in $\mathfrak{C}_G$ has at most one neighbour in $V(A_G)$.
That means that for every $\pi \in \Aut(D_{G,S})$, we have 
\[  \pi(V(A_G)) = V(A_G) \textnormal{ and } \pi(V(\mathfrak{C}_G)) = V(\mathfrak{C}_G). \]
Furthermore, we see that the only triangles in our graph $D_{G,S}$ are given by cycles of the form \[ \left( \mathfrak{C}_{[d,g,h]}(0), \mathfrak{C}_{[d,g,h]}(1), \mathfrak{C}_{[d,g,h]}(3) \right),\] where $(g,h) \in E(\Cay_{G,S})$ is of colour $d$. Thus, these three vertices can only be mapped setwise to one another, since automorphisms respect types.
In the following, we show that the image of a $d$-chain in $D_{G,S}$ under an automorphism is again a $d$-chain.

\begin{lemma}\label{lemmaDChainsPermus}
    Let $g,h\in G$ be elements such that $\mathfrak{C}_{[d,g,h]}$ forms a $d$-chain in $D_{G,S}$ and an automorphism $\pi \in \Aut(D_{G,S})$. Then, there are unique $g^{\prime}, h^{\prime} \in G$ such that
    \[ \pi(\mathfrak{C}_{[d,g,h]}(v)) = \mathfrak{C}_{[d,g^{\prime}, h^{\prime}]}(v)  \]
    for all $v \in V(\mathfrak{C}_{[d,g,h]})$. In particular, $\mathfrak{C}_{[d,g^{\prime}, h^{\prime}]}$ also forms a $d$-chain in $D_{G,S}$.
\end{lemma}
\begin{proof}
    It suffices to prove the above statement only for the endgadget of the $d$-chain. The proposed equality for the remaining vertices of the $d$-chain $\mathfrak{C}_{[d,g,h]}$ then follows by exploiting the incidence structure of the graph $D_{G,S}$. Clearly, the endgadget of $\mathfrak{C}_{[d,g,h]}$ must be mapped setwise to another endgadget. Hence, there exist $g',h'\in G$ and $d'\in [k]$ such that
    \[ \pi(\mathfrak{C}_{[d,g,h]}(j)) = \mathfrak{C}_{[d^{\prime}, g^{\prime}, h^{\prime}]} (j^{\prime}) \]
    for every $j \in \{0, \ldots, 9\}$, where $j^{\prime} \in \{0, \ldots, 9\}$.
    We already know that the vertices $\mathfrak{C}_{[d,g,h]}(0)$, $\mathfrak{C}_{[d,g,h]}(1)$ and $ \mathfrak{C}_{[d,g,h]}(3)$ must be mapped setwise to $\{\mathfrak{C}_{[d',g',h']}(i)\mid i=0,1,3\}$. Since the type of $ \mathfrak{C}_{[d,g,h]}(1)$ differs from the types of $ \mathfrak{C}_{[d,g,h]}(0)$ and $ \mathfrak{C}_{[d,g,h]}(3)$, see \cref{table:types}, we deduce that $\pi( \mathfrak{C}_{[d,g,h]}(1))= \mathfrak{C}_{[d',g',h']}(1)$. 
    The two vertices $\mathfrak{C}_{[d,g,h]}(6)$ and $\mathfrak{C}_{[d,g,h]}(9)$ are the only vertices of the endgadget connected to vertices of a blow-up graph. Since the types of the blow-up vertices and the types of the chain vertices differ, we know that these two vertices are mapped setwise to one another. However, $\mathfrak{C}_{[d,g,h]}(6)$ has a neighbour of type $(3,6,7)$, namely the vertex $\mathfrak{C}_{[d,g,h]}(1)$, while $\mathfrak{C}_{[d,g,h]}(9)$ has no neighbour of this type. Thus, we know that
    \begin{equation}\label{eq1}
        \pi(\mathfrak{C}_{[d,g,h]}(6)) = \mathfrak{C}_{[d^{\prime},g^{\prime},h^{\prime}]}(6) \textnormal{ and } \pi(\mathfrak{C}_{[d,g,h]}(9)) = \mathfrak{C}_{[d^{\prime},g^{\prime},h^{\prime}]}(9)
    \end{equation}
    for some $d^{\prime} \in [k]$ and $g^{\prime}, h^{\prime} \in G$.
    Since $\mathfrak{C}_{[d,g,h]}(9)$ is mapped onto $\mathfrak{C}_{[d',g',h']}(9)$ under $\pi$, we know that $\pi(\mathfrak{C}_{[d,g,h]}(8))=\mathfrak{C}_{[d',g',h']}(8)$. It now follows easily that all gadget vertices are determined, that is
    \[ \pi(\mathfrak{C}_{[d,g,h]}(j)) = \mathfrak{C}_{[d^{\prime},g^{\prime},h^{\prime}]}(j) \]
    for every $j \in \{0, \ldots, 9\}$ and some fixed $d^{\prime} \in [k]$ and $ g', h' \in G$.

It remains to show that $d = d^{\prime}$.
    Because of \cref{eq1}, we conclude that 
    \[  \pi(\mathfrak{C}_{[d,g,h]}(u_d)) =  \mathfrak{C}_{[d^{\prime},g^{\prime},h^{\prime}]}(u_d) \textnormal{ and } \pi(\mathfrak{C}_{[d,g,h]}(\ell_d)) =  \mathfrak{C}_{[d^{\prime},g^{\prime},h^{\prime}]}(\ell_d).
    \]
    By induction on the length of the $d$-chain, the above equations follow for all upper and lower chain vertices. If $d \neq d^{\prime}$, this chain of equations would fail at some point. Thus, $d = d^{\prime}$ and therefore $g^{\prime}$ and $h^{\prime}$ are unique.
\end{proof}
Now, we show a similar uniqueness result on the permutations of the blow-up graphs. We again limit the statement to the case $k \geq 3$. The other two cases can be proven analogously.
\begin{lemma}\label{lemmaDNGPermus}
    Let $g \in G$ and $A_g$ be the blow-up graph of $g$, and $\pi \in \Aut(D_{G,S})$. Then, there is a unique $g' \in G$ such that 
    \[  \pi( a_g^{d.i}) = a_{g'}^{d.i} \textnormal{ and } \pi(c_g^{d.j}) = c_{g'}^{d.j} \]
    for all $d \in [k], i \in [4]$ and $j \in [2]$.
\end{lemma}
\begin{proof}
    Since center vertices have no $d$-chain neighbours, we see that every center vertex is mapped onto another center vertex and every connector vertex is mapped onto a connector vertex under $\pi$. Hence, we have
    \[  \pi(a_g^{d.i}) = a_{g'}^{d^{\prime}.i^{\prime}} \textnormal{ and } \pi(c_g^{d.i}) = c_{{g'}^{\prime}}^{d^{\prime \prime}.i^{\prime \prime}}. \]

    By \cref{lemmaDChainsPermus} we know that if a $d$-chain $\mathfrak{C}_{[d,g,h]}$ is mapped to another $d$-chain $\mathfrak{C}_{[d, g^{\prime}, h^{\prime}]}$, it respects the position of every vertex between both chains. Therefore, the vertex $a_g^{d.i}$ can only be sent to the vertex $a_{g'}^{d.i}$. The reasoning applies for both the incoming and outgoing $d$-chains.
    Thus, since all connector vertices are fixed within a blow-up graph, and likewise all center vertices, $g'$ is unique.
\end{proof}
Given \cref{lemmaDChainsPermus,lemmaDNGPermus} for the $d$-chains and the blow-up graphs, we say that $d$-chains and blow-up graphs  are mapped ``as a whole'' to $d$-chains and blow-up graphs.
Finally, we have the following lemma.
\begin{lemma}\label{lemmaAutDLeftAction}
    Let $\pi \in \Aut(D_{G,S})$. If the blow-up graph of $1_G$ is mapped as a whole to the blow-up graph $A_h$ for $h \in G$, then, for every $g \in G$, $\pi$ maps the blow-up graph $A_g$ of $g$ as a whole to the blow-up graph $A_{hg}$.
\end{lemma}
\begin{proof}
The proof follows from examining the $d$-chains in $D_{G,S}.$
  We know that $\pi(A_{1_G}) = A_h$. Since $1_Gs_d = s_d$ for all $d \in [k]$, we have a $d$-chain $\mathfrak{C}_{[d,1_G, s_d]}$ for all $d \in [k]$. \cref{lemmaDChainsPermus} implies that for all those outgoing $d$-chains of $A_{1_G}$, we have
  \[ \pi(\mathfrak{C}_{[d,1_G, s_d]}) = \mathfrak{C}_{[d, h, f]}, \]
  where $f = hs_d$. Now, \cref{lemmaDChainsPermus,lemmaDNGPermus} imply that $\pi(A_{s_d}) = A_{hs_d}$ for all $d \in [k]$. Moreover, if $\pi(A_g) = A_f$ for some $f \in G$, we know that $\pi(A_{gs_d}) = A_{fs_d}$ for every generator $s_d \in S$. This again is an immediate consequence of \cref{lemmaDChainsPermus,lemmaDNGPermus}. Given these two facts, we can now argue our claim for all $g \in G$ by induction on the length of a word in the generators yielding the element $g \in G$. Let $g = s_{i_1} \cdot \; \cdots \; \cdot \; s_{i_r}$, and consider a $d$-chain $\mathfrak{C}_{[d,g,gs_d]}$. Then,
  \begin{align*}
      \pi(\mathfrak{C}_{[d,g,gs_d]}) & = \pi(\mathfrak{C}_{[d,(s_{i_1} \cdot \;  \cdots \; \cdot \; s_{i_r}),(s_{i_1} \cdot \;  \cdots \; \cdot \; s_{i_r})s_d]}) \\
      &= \pi(\mathfrak{C}_{[d,(s_{i_1} \cdot \;  \cdots \; \cdot \; s_{i_{r-1}}) s_{i_r},(s_{i_1} \cdot \;  \cdots \; \cdot \; s_{i_{r-1}} ) s_{i_r}s_d]}) \\
      & = \ldots \\
      & = \pi(\mathfrak{C}_{[d,s_{i_1} (s_{i_2} \cdot \;  \cdots \; \cdot \; s_{i_r}), s_{i_1} (s_{i_2} \cdot \;  \cdots \; \cdot \;s_{i_r}s_d)]}) \\
      & = \mathfrak{C}_{[d,hs_{i_1} (s_{i_2} \cdot \; \cdots \; \cdot \; s_{i_r}),hs_{i_1} (s_{i_2} \cdot \; \cdots \; \cdot \; s_{i_r}) s_d]} \\
      & = \mathfrak{C}_{[d,hg,hgs_d]},
  \end{align*}
  and the claim is proven by invoking \cref{lemmaDChainsPermus,lemmaDNGPermus}.
\end{proof}

Finally, we confirm that every automorphism of $D_{G,S}$ is already induced by an automorphism of the initial Cayley graph.
\begin{proof}[Proof of \cref{lemmaAutDAutCayley}]
    \cref{lemmaAutDLeftAction} implies that every blow-up graph is mapped as a whole to another blow-up graph according to the left multiplication with one fixed element $h \in G$.
    \cref{lemmaDChainsPermus,lemmaDNGPermus} now imply that every incoming and outgoing $d$-chain of a blow-up graph $A_g$ for $g \in G$ is mapped to the incoming or outgoing $d$-chain of $A_{hg}$.
    Therefore, by contracting every blow-up graph $A_g$ for $g \in G$ to a single vertex $g$, and by reducing all $d$-chains to single, directed edges of colour $d \in [k]$, we obtain a unique automorphism $\pi \in \Aut(\Cay_{G,S}).$
\end{proof}

Clearly, the map from \cref{lemmaAutCayleyAutD} sending automorphisms of the edge-coloured Cayley graph to automorphisms of $D_{G,S}$ factorizes over composition of maps. Thus, we have a group homomorphism from $\Aut(\Cay_{G,S})$ to $\Aut(D_{G,S})$. \cref{lemmaAutCayleyAutD,lemmaAutDAutCayley} imply the one-to-one correspondence between the two automorphism groups. Thus, we can finally ensure that the automorphism group of the graph $D_{G,S}$ is isomorphic to the initial group $G$. 

\begin{theorem}\label{theoremAutGroupD}
    Let $G$ be a group and $S$ a generating set for $G$. Then, we have $\Aut(D_{G,S}) \cong G$.
\end{theorem}

\section{Constructing a Cycle Double Cover}\label{sec1CutCDC}
In this section, we construct a cycle double cover (CDC) of the 3-connected cubic graph $D_{G,S}$ that induces a polyhedral map. 
In particular, we construct a set of cycles $\mathcal{Z}$ in $D_{G,S}$ such that any two distinct cycles intersect in at most one edge. 
For this, we present three types of cycles $\mathcal{Z}_i, \mathcal{Z}_r$ and $\mathcal{Z}_b$ such that $\mathcal{Z}=\mathcal{Z}_i\cup\mathcal{Z}_r\cup\mathcal{Z}_b$ is satisfied.
Cycles in $\mathcal{Z}_i$ will be referred to as \textbf{inner face cycles}, those in $\mathcal{Z}_r$ as \textbf{roof edge cycles}, and cycles in $\mathcal{Z}_b$ as \textbf{base edge cycles}. 
Throughout this section, let $G$ be a finite group together with a generating set $S=\{s_1, \ldots, s_k\}$ and an ordering $s_1 < \cdots < s_k$ of its generators. Based on the choice of generators and their ordering, we consider the graph $D_{G,S}$ from \cref{defGraphD}.

We recommend the reader to thoroughly inspect the figures in this section, in particular \cref{fig:InteriorCyclesAg,fig:cyclesD10,fig:cdcUpdated,fig:cyclesD10Updated,fig:cdcUpdated2,fig:cyclesD10Updated1}, which illustrate the cycles and their intersections of our CDC $\mathcal{Z}$. In fact, these illustrations precisely reflect our implementation of the cycles in \cite{gitBA}. We prove the correctness and completeness of our constructed CDC in \cref{secProperties}.

\subsection{The Inner Face Cycles}
First, we construct the inner face cycles. 
Instead of providing a formal definition of the cycles, we describe the cycles and illustrate them graphically in \cref{fig:InteriorCyclesAg,fig:cyclesD10}.
Note that for all $g,h\in G$ and every colour $d \in [k]$, both the $d$-chains $\mathfrak{C}_{[d,g,h]}$ and the blow-up graphs $A_g$ are planar. 
Nevertheless, the graph construction introduced in the previous section does not necessarily yield a planar graph, even when the underlying Cayley graph is planar.
Since both the $d$‑chains and the blow-up graphs admit planar embeddings, we define the boundaries of the inner faces of these planar embeddings as the cycles of $\mathcal{Z}_i$.
It is straightforward to verify that any two cycles in $\mathcal{Z}_i$ intersect in at most one edge. In addition, every edge in $D_{G,S}$ is contained in at least one cycle of $\mathcal{Z}_i$.
Note that the exterior green faces delimited by the dashed edges in \cref{fig:InteriorCyclesAg} correspond to the green faces in \cref{fig:cyclesD10} between the $d$-chains and the blow-up graphs. 
The thick light blue edges in these figures are the edges that are already contained in two cycles of $\mathcal{Z}_i$. 

Note that the term ``colour'' is used in two distinct ways: it refers both to the $k$ colours assigned to the generators $s_1, \ldots, s_k$, and to the visual colours used to highlight faces or edge sets in our cycle double cover. 
In this section, these visual colours for the faces or edges are entirely unrelated to the generator colours. We always use the integers $\{1, \ldots, k\}$ for the edge colours of the underlying Cayley graph, and visible colours solely to membership in the sets $\mathcal{Z}_i$.

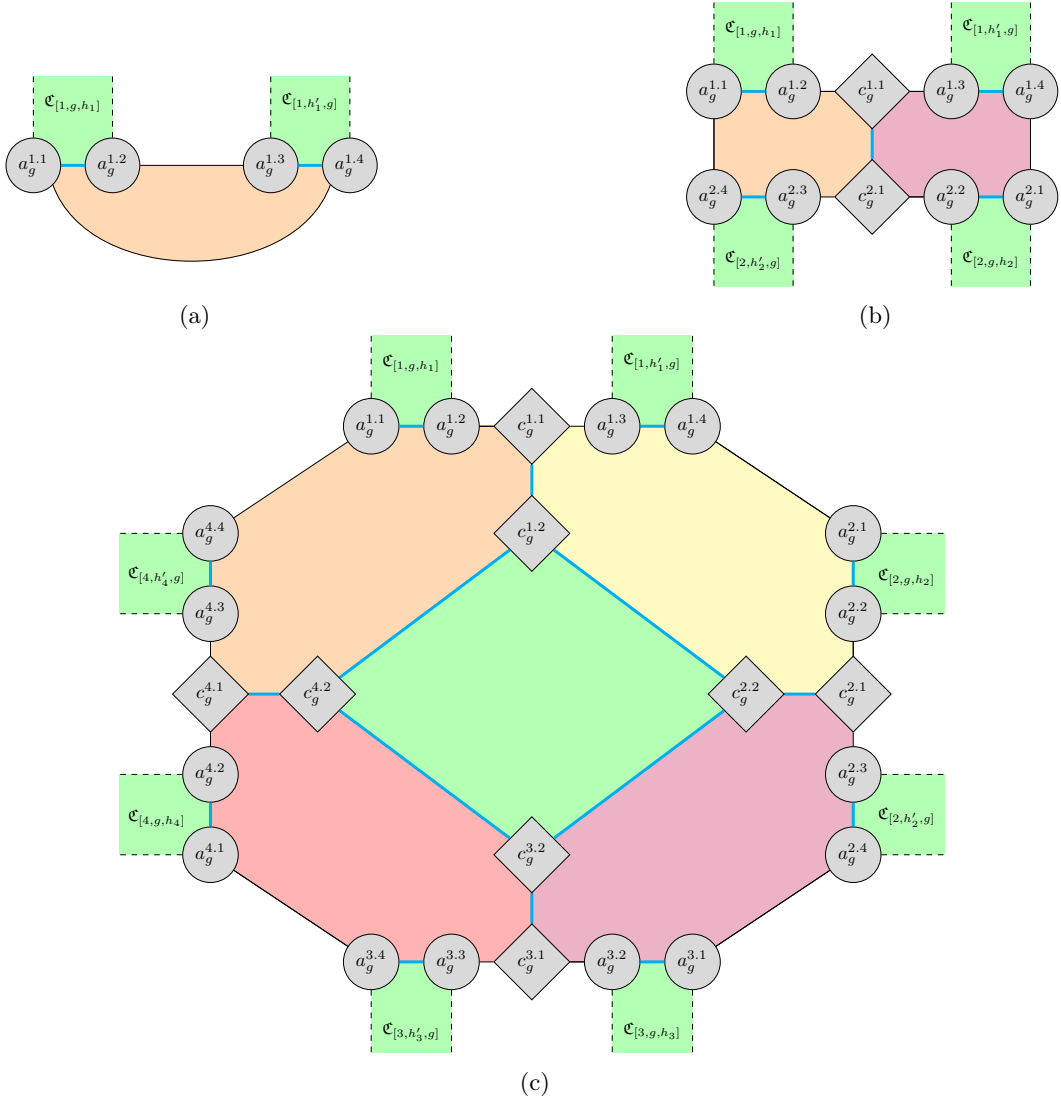
\begin{figure}[ht]
\centering
\begin{subfigure}{0.4\textwidth}
    \centering
    \resizebox{5cm}{!}{\begin{tikzpicture}[
    node distance = 3mm and 3mm,
         V/.style = {circle, draw, fill=gray!30},
    every edge quotes/.style = {auto, font=\footnotesize, sloped}
                        ]

                        \pgfdeclarelayer{background}
                        \pgfsetlayers{background,main}

\tikzstyle{overbrace text style}=[font=\large, above, pos=.5, yshift=4mm]
\tikzstyle{underbrace text style}=[font=\large, below, pos=.5, yshift=4mm]
\tikzstyle{overbrace style}=[decorate,decoration={brace,raise=3mm,amplitude=3pt}]
\tikzstyle{underbrace style}=[decorate,decoration={brace,raise=3mm,amplitude=3pt,mirror}]

\coordinate (1) at (3,1);
\coordinate (2) at (4.5, 1);
\coordinate (3) at (7.5,1);
\coordinate (4) at (9,1);

\path [fill=green!30 ] (1) rectangle (4.5, 2.7);
\path [fill=green!30 ] (3) rectangle (9, 2.7);

\begin{scope}[nodes=V]

\node[circle, draw] (7) at (3,1)   {$a_g^{1.1}$};
\node[circle, draw] (8) at (4.5,1)   {$a_g^{1.2}$};
\node[circle, draw] (11) at (7.5,1)   {$a_g^{1.3}$};
\node[circle, draw] (12) at (9,1)   {$a_g^{1.4}$};

\end{scope}

\begin{pgfonlayer}{background}
\filldraw[orange!30]
        (7.center) -- (8.center) -- (11.center) -- (12.center)  -- (12.south west) to[bend left = 70] (7.south east) -- cycle;
\end{pgfonlayer}

\draw

        (7) edge[dashed] (3, 2.7)
        (8) edge[dashed] (4.5, 2.7)
        (11) edge[dashed] (7.5, 2.7)
        (12) edge[dashed] (9, 2.7)

        (7) edge[ultra thick, cyan] (8)
        (8) edge[] (11)
        (11) edge[ultra thick, cyan] (12)
        (7.south east) edge[bend right = 70] (12.south west);

\draw [draw = none] (3, 1.5) -- (4.5, 1.5) node [overbrace text style, font=\small] {$\mathfrak{C}_{[1,g,h_1]}$};

\draw [draw = none] (7.5, 1.5) -- (9, 1.5) node [overbrace text style, font=\small] {$\mathfrak{C}_{[1,h_1^{\prime}, g]}$};

        \end{tikzpicture}}
    \caption{}
    \label{fig:cyclesD1gen}
\end{subfigure}
\hfill
\begin{subfigure}[b]{0.4\textwidth}
    \centering
    \resizebox{5cm}{!}{\begin{tikzpicture}[
    node distance = 3mm and 3mm,
         V/.style = {circle, draw, fill=gray!30},
    every edge quotes/.style = {auto, font=\footnotesize, sloped}
                        ]

\tikzstyle{overbrace text style}=[font=\large, above, pos=.5, yshift=4mm]
\tikzstyle{underbrace text style}=[font=\large, below, pos=.5, yshift=4mm]
\tikzstyle{overbrace style}=[decorate,decoration={brace,raise=3mm,amplitude=3pt}]
\tikzstyle{underbrace style}=[decorate,decoration={brace,raise=3mm,amplitude=3pt,mirror}]

\coordinate (1) at (9, -1);
\coordinate (2) at (7.5, -1);
\coordinate (3) at (6, -1);
\coordinate (4) at (4.5 -1);
\coordinate (5) at (3, -1);
\coordinate (10) at (9, 1);
\coordinate (9) at (7.5, 1);
\coordinate (8) at (6, 1);
\coordinate (7) at (4.5, 1);
\coordinate (6) at (3, 1);

\path [fill=green!30 ] (1) -- (2) -- (7.5, -2.7) -- (9,-2.7);
\path [fill=green!30 ] (5) rectangle (4.5, -2.7);
\path [fill=green!30 ] (6) rectangle (4.5, 2.7);
\path [fill=green!30 ] (9) rectangle (9, 2.7);

\draw[fill=orange!30] (3)--(8)--(7)--(6)--(5);
\draw[fill=purple!30] (1)--(2)--(3)--(8)--(9)--(10);

\begin{scope}[nodes=V]

\node[circle, draw] (19) at (9,-1)   {$a_g^{2.1}$};
\node[circle, draw] (20) at (7.5,-1)   {$a_g^{2.2}$};
\node[shape = diamond, draw] (21) at (6,-1)   {$c_g^{2.1}$};
\node[circle, draw] (23) at (4.5,-1)   {$a_g^{2.3}$};
\node[circle, draw] (24) at (3,-1)   {$a_g^{2.4}$};

\node[circle, draw] (7) at (3,1)   {$a_g^{1.1}$};
\node[circle, draw] (8) at (4.5,1)   {$a_g^{1.2}$};
\node[shape = diamond, draw] (9) at (6,1)   {$c_g^{1.1}$};
\node[circle, draw] (11) at (7.5,1)   {$a_g^{1.3}$};
\node[circle, draw] (12) at (9,1)   {$a_g^{1.4}$};

\end{scope}

\draw

        (7) edge[dashed] (3, 2.7)
        (8) edge[dashed] (4.5, 2.7)
        (11) edge[dashed] (7.5, 2.7)
        (12) edge[dashed] (9, 2.7)
        (9) edge[ultra thick, cyan] (21)
        (12) edge (19)
        (24) edge (7)

        (24) edge[dashed] (3, -2.7)
        (23) edge[dashed] (4.5, -2.7)
        (20) edge[dashed] (7.5, -2.7)
        (19) edge[dashed] (9, -2.7)

        (7) edge[ultra thick, cyan] (8)
        (8) edge[] (9)
        (9) edge[] (11)
        (11) edge[ultra thick, cyan] (12)

        (19) edge[ultra thick, cyan] (20)
        (20) edge[] (21)
        (21) edge[] (23)
        (23) edge[ultra thick, cyan] (24);

\draw [draw = none] (3, 1.5) -- (4.5, 1.5) node [overbrace text style, font=\small] {$\mathfrak{C}_{[1,g,h_1]}$};
\draw [draw = none] (3, -2.3) -- (4.5, -2.3) node [underbrace text style, font=\small] {$\mathfrak{C}_{[2,h_2^{\prime},g]}$};

\draw [draw = none] (7.5, 1.5) -- (9, 1.5) node [overbrace text style, font=\small] {$\mathfrak{C}_{[1,h_1^{\prime},g]}$};
\draw [draw = none] (7.5, -2.3) -- (9, -2.3) node [underbrace text style, font=\small] {$\mathfrak{C}_{[2, g,h_2]}$};

\end{tikzpicture}}
     \caption{}
    \label{fig:agCycles2Gens_inner}
\end{subfigure}
\begin{subfigure}{0.99\textwidth}
    \centering
    \resizebox{11cm}{!}{\begin{tikzpicture}[
    node distance = 3mm and 3mm,
         V/.style = {circle, draw, fill=gray!30},
    every edge quotes/.style = {auto, font=\footnotesize, sloped}
                        ]

\tikzstyle{overbrace text style}=[font=\large, above, pos=.5, yshift=4mm]
\tikzstyle{underbrace text style}=[font=\large, below, pos=.5]
\tikzstyle{overbrace style}=[decorate,decoration={brace,raise=3mm,amplitude=3pt}]

\coordinate (1) at (0,0);
\coordinate (2) at (0, -1.5);
\coordinate (3) at (0, -3);
\coordinate (4) at (2, -3);
\coordinate (5) at (0, -4.5);
\coordinate (6) at (0, -6);

\coordinate (7) at (3,2);
\coordinate (8) at (4.5, 2);
\coordinate (9) at (6, 2);
\coordinate (10) at (6, 0);
\coordinate (11) at (7.5, 2);
\coordinate (12) at (9, 2);

\coordinate (13) at (12, 0);
\coordinate (14) at (12, -1.5);
\coordinate (15) at (12, -3);
\coordinate (16) at (10, -3);
\coordinate (17) at (12, -4.5);
\coordinate (18) at (12, -6);

\coordinate (19) at (9, -8);
\coordinate (20) at (7.5, -8);
\coordinate (21) at (6, -8);
\coordinate (22) at (6, -6);
\coordinate (23) at (4.5, -8);
\coordinate (24) at (3, -8);

\draw[fill=orange!30] (1)--(2)--(3)--(4)--(10)--(9)--(8)--(7);


\path [fill=green!30 ] (1) -- (2) -- (-1.7, -1.5) -- (-1.7,0);
\path [fill=green!30 ] (5) -- (6) -- (-1.7, -6) -- (-1.7,-4.5);
\path [fill=green!30 ] (7) -- (8) -- (4.5, 3.7) -- (3,3.7);
\path [fill=green!30 ] (11) -- (12) -- (9, 3.7) -- (7.5,3.7);
\path [fill=green!30 ] (23) -- (24) -- (3, -9.7) -- (4.5,-9.7);
\path [fill=green!30 ] (19) -- (20) -- (7.5, -9.7) -- (9,-9.7);
\path [fill=green!30 ] (13) -- (14) -- (13.7, -1.5) -- (13.7,0);
\path [fill=green!30 ] (17) -- (18) -- (13.7, -6) -- (13.7,-4.5);

\draw[fill=yellow!30] (9)--(10)--(16)--(15)--(14)--(13)--(12)--(11);

\draw[fill=purple!30] (15)--(16)--(22)--(21)--(20)--(19)--(18)--(17);

\draw[fill=red!30] (21)--(22)--(4)--(3)--(5)--(6)--(24)--(23);

\draw[fill=green!30] (4)--(10)--(16)--(22);

\begin{scope}[nodes=V]

\node[circle, draw] (13) at (12,0)   {$a_g^{2.1}$};
\node[circle, draw] (14) at (12,-1.5)   {$a_g^{2.2}$};
\node[shape = diamond, draw] (15) at (12,-3)   {$c_g^{2.1}$};
\node[shape = diamond, draw] (16) at (10,-3)   {$c_g^{2.2}$};
\node[circle, draw] (17) at (12, -4.5)   {$a_g^{2.3}$};
\node[circle, draw] (18) at (12,-6)   {$a_g^{2.4}$};

\node[circle, draw] (19) at (9,-8)   {$a_g^{3.1}$};
\node[circle, draw] (20) at (7.5,-8)   {$a_g^{3.2}$};
\node[shape = diamond, draw] (21) at (6,-8)   {$c_g^{3.1}$};
\node[shape = diamond, draw] (22) at (6,-6)   {$c_g^{3.2}$};
\node[circle, draw] (23) at (4.5,-8)   {$a_g^{3.3}$};
\node[circle, draw] (24) at (3,-8)   {$a_g^{3.4}$};

\node[circle, draw] (1) at (0,0)   {$a_g^{4.4}$};
\node[circle, draw] (2) at (0,-1.5)   {$a_g^{4.3}$};
\node[shape = diamond, draw] (3) at (0,-3)   {$c_g^{4.1}$};
\node[shape = diamond, draw] (4) at (2,-3)   {$c_g^{4.2}$};
\node[circle, draw] (5) at (0,-4.5)   {$a_g^{4.2}$};
\node[circle, draw] (6) at (0,-6)   {$a_g^{4.1}$};

\node[circle, draw] (7) at (3,2)   {$a_g^{1.1}$};
\node[circle, draw] (8) at (4.5,2)   {$a_g^{1.2}$};
\node[shape = diamond, draw] (9) at (6,2)   {$c_g^{1.1}$};
\node[shape = diamond, draw] (10) at (6,0)   {$c_g^{1.2}$};
\node[circle, draw] (11) at (7.5,2)   {$a_g^{1.3}$};
\node[circle, draw] (12) at (9,2)   {$a_g^{1.4}$};

\end{scope}

\draw

        (7) edge[dashed] (3, 3.7)
        (8) edge[dashed] (4.5, 3.7)
        (11) edge[dashed] (7.5, 3.7)
        (12) edge[dashed] (9, 3.7)

        (24) edge[dashed] (3, -9.7)
        (23) edge[dashed] (4.5, -9.7)
        (20) edge[dashed] (7.5, -9.7)
        (19) edge[dashed] (9, -9.7)

        (1) edge[dashed] (-1.7, 0)
        (2) edge[dashed] (-1.7, -1.5)
        (5) edge[dashed] (-1.7, -4.5)
        (6) edge[dashed] (-1.7, -6) 

        (13) edge[dashed] (13.7, 0)
        (14) edge[dashed] (13.7, -1.5)
        (17) edge[dashed] (13.7, -4.5)
        (18) edge[dashed] (13.7, -6) 

        (1) edge[ultra thick, cyan] (2)
        (2) edge[] (3)
        (3) edge[ultra thick, cyan] (4)
        (3) edge[] (5)
        (5) edge[ultra thick, cyan] (6)

        (4) edge[ultra thick, cyan] (10)
        (10) edge[ultra thick, cyan] (16)
        (16) edge[ultra thick, cyan] (22)
        (22) edge[ultra thick, cyan] (4)

        (1) edge[] (7)
        (12) edge[] (13)
        (18) edge[] (19)
        (24) edge[] (6)

        (7) edge[ultra thick, cyan] (8)
        (8) edge[] (9)
        (9) edge[ultra thick, cyan] (10)
        (9) edge[] (11)
        (11) edge[ultra thick, cyan] (12)

        (13) edge[ultra thick, cyan] (14)
        (14) edge[] (15)
        (15) edge[ultra thick, cyan] (16)
        (15) edge[] (17)
        (17) edge[ultra thick, cyan] (18)

        (19) edge[ultra thick, cyan] (20)
        (20) edge[] (21)
        (21) edge[ultra thick, cyan] (22)
        (21) edge[] (23)
        (23) edge[ultra thick, cyan] (24);

        \draw [draw = none] (3, 2.5) -- (4.5, 2.5) node [overbrace text style, font=\small] {$\mathfrak{C}_{[1,g,h_1]}$};
\draw [draw = none] (3, -9) -- (4.5, -9) node [underbrace text style, font=\small] {$\mathfrak{C}_{[3,h_3^{\prime}, g]}$};

\draw [draw = none] (7.5, 2.5) -- (9, 2.5) node [overbrace text style, font=\small] {$\mathfrak{C}_{[1,h_1^{\prime}, g]}$};
\draw [draw = none] (7.5, -9) -- (9, -9) node [underbrace text style, font=\small] {$\mathfrak{C}_{[3,g,h_3]}$};

\draw [draw = none] (-1.5, -6) -- (-0.5, -6) node [overbrace text style, font=\small] {$\mathfrak{C}_{[4,g,h_4]}$};
\draw [draw = none] (13.5, -6) -- (12.5, -6) node [overbrace text style, font=\small] {$\mathfrak{C}_{[2,h_2^{\prime}, g]}$};

\draw [draw = none] (-1.5, -1.5) -- (-0.5, -1.5) node [overbrace text style, font=\small] {$\mathfrak{C}_{[4,h_4^{\prime}, g]}$};
\draw [draw = none] (13.5, -1.5) -- (12.5, -1.5) node [overbrace text style, font=\small] {$\mathfrak{C}_{[2,g,h_2]}$};

        \end{tikzpicture}}
     \caption{}
    \label{fig:agCycles3Gens_inner}
\end{subfigure}
\caption{The inner face cycles of $A_g$ for $k=1$ (a), $k=2$ (b) and $k=4$ (c).}
\label{fig:InteriorCyclesAg}
\end{figure}

\begin{figure}[ht!]
    \centering
    \resizebox{9cm}{!}{
    \begin{tikzpicture}[
    node distance = 3mm and 3mm,
         V/.style = {circle, draw, fill=gray!30},
    every edge quotes/.style = {auto, font=\footnotesize, sloped}
                        ]

\tikzstyle{overbrace text style}=[font=\large, above, pos=.5, yshift=4mm]
\tikzstyle{underbrace text style}=[font=\large, below, pos=.5, yshift=-4mm]
\tikzstyle{overbrace style}=[decorate,decoration={brace,raise=3mm,amplitude=3pt}]
\tikzstyle{underbrace style}=[decorate,decoration={brace,raise=3mm,amplitude=3pt,mirror}]

\coordinate (1) at (0,0);
\coordinate (2) at (4,0);
\coordinate (3) at (4, -1.5);
\coordinate (4) at (0, -1.5);

\coordinate (5) at (7, 0);
\coordinate (6) at (7, -1.5);
\coordinate (5a) at (8.1, 0);
\coordinate (6a) at (8.1, -1.5);

\coordinate (7) at (0, -4.5);
\coordinate (8) at (0, -6);
\coordinate (9) at (4, -4.5);
\coordinate (10) at (4, -6);

\coordinate (11) at (5.5, -4.5);
\coordinate (12) at (5.5, -7);
\coordinate (13) at (4, -8);

\coordinate (a) at (8.1, -4.5);
\coordinate (b) at (8.1, -6);

\coordinate (14) at (7, -8);

\coordinate (15) at (7, -4.5);
\coordinate (16) at (7, -6);

\filldraw[red!30] (15)--(a)--(b)--(16);
\filldraw[green!30] (5)--(5a)--(6a)--(3);
\filldraw[green!30] (1)--(2)--(3)--(4);
\filldraw[red!30] (2)--(5)--(6)--(3);
\filldraw[green!30] (7)--(9)--(10)--(8);
\filldraw[brown!30] (9)--(10)--(13)--(12)--(11);
\filldraw[purple!30] (11)--(12)--(14)--(16)--(15);
\filldraw[orange!30] (12)--(13)--(14);

\filldraw[green!15] (7, -7)--(5.5, -6)--(5.5, -7)-- (7, -8);

\begin{scope}[nodes=V]

\node (a) at (4, 0)    {$u_1$};
\node (b) at (4, -1.5)          {$\ell_1$};

\node (c) at (7, 0) {$u_2$};
\node (d) at (7, -1.5) {$\ell_2$};


\node (a8) at (4, -4.5) {9};
\node (a5) at (4, -6) {6};
\node (a7) at (5.5, -4.5) {8};
\node (a6) at (7, -4.5) {7};
\node (a4) at (7, -6) {4};
\node (a3) at (5.5, -7) {3};
\node (a2) at (4, -8) {1};
\node (a1) at (7, -8) {$0$};
\node (a0) at (7, -7) {$2$};
\node (a9) at (5.5, -6) {$5$};

\node[circle, draw] (1) at (0,0)   {$a_g^{d.1}$};
\node[circle, draw] (2) at (0,-1.5)   {$a_g^{d.2}$};
\node[shape = diamond] (3) at (0,-3)   {$c_g^{d.1}$};
\node[shape = diamond] (4) at (-2,-3)   {$c_g^{d.2}$};
\node[circle, draw] (5) at (0,-4.5)   {$a_g^{d.3}$};
\node[circle] (6) at (0,-6)   {$a_g^{d.4}$};

\node(27) at (-1, 1.5) {$a_g^{(d-1).4}$};
\node (28) at (-1, -7.5) {$a_g^{(d+1).1}$};

\node[shape = diamond] (25) at (-3, -1.5) {$c_g^{(d-1).2}$};
\node[shape = diamond] (26) at (-3, -4.5) {$c_g^{(d+1).2}$};

\end{scope}

\draw   

        (1) edge (a)
        (2) edge (b)

        (a) edge (c)
        (b) edge (d)
        (c) edge[ultra thick, cyan] (d)
        (a) edge[ultra thick, cyan] (b)

        (c) edge[dashed] (8.2,0)
        (d) edge[dashed] (8.2,-1.5)

        (a4) edge[dashed] (8.1, -6)
        (a6) edge[dashed] (8.1, -4.5)

        (5) edge (a8)
        (6) edge (a5)
        
        (27) edge[dashed, ultra thick, cyan] (-2, 2.5)
        (27) edge[dashed] (0, 2.5)

        (28) edge[dashed, ultra thick, cyan] (-2, -8.5)
        (28) edge[dashed] (0, -8.5)

        (25) edge[dashed, ultra thick, cyan] (-4,0)
        (25) edge[dashed, ultra thick, cyan] (-2, 0)

        (26) edge[dashed, ultra thick, cyan] (-4, -6)
        (26) edge[dashed, ultra thick, cyan] (-2, -6)

        (1) edge [] (27)
        (6) edge[] (28)

        (4) edge[ultra thick, cyan] (25)
        (4) edge[ultra thick, cyan] (26)

        (1) edge[ultra thick, cyan] (2)
        (2) edge (3)
        (3) edge[ultra thick, cyan] (4)
        (3) edge (5)
        (5) edge[ultra thick, cyan] (6)

        (a1) edge (a2)
        (a1) edge[ultra thick, cyan] (a3)
        (a1) edge (a0)
        (a4) edge (a0)
        (a2) edge (a5)
        (a2) edge[ultra thick, cyan] (a3)
        (a3) edge[ultra thick, cyan] (a9)
        (a9) edge[ultra thick, cyan] (a7)

        (a9) edge[ultra thick, cyan] (a0)
        
        (a4) edge[ultra thick, cyan] (a6)
        (a5) edge[ultra thick, cyan] (a8)
        (a7) edge (a6)
        (a7) edge (a8);

\draw [overbrace style] (3.5, 0.5) -- (8.3, 0.5) node [overbrace text style] {The chain $\mathfrak{C}_{[d,g,h]}$};

\draw [underbrace style] (3.5, -8.5) -- (8.3, -8.5) node [underbrace text style] {The chain $\mathfrak{C}_{[d,h^{\prime},g]}$};

\end{tikzpicture}
    }
    \caption{The inner face cycles in a $d$-chain.}
    \label{fig:cyclesD10}
\end{figure}
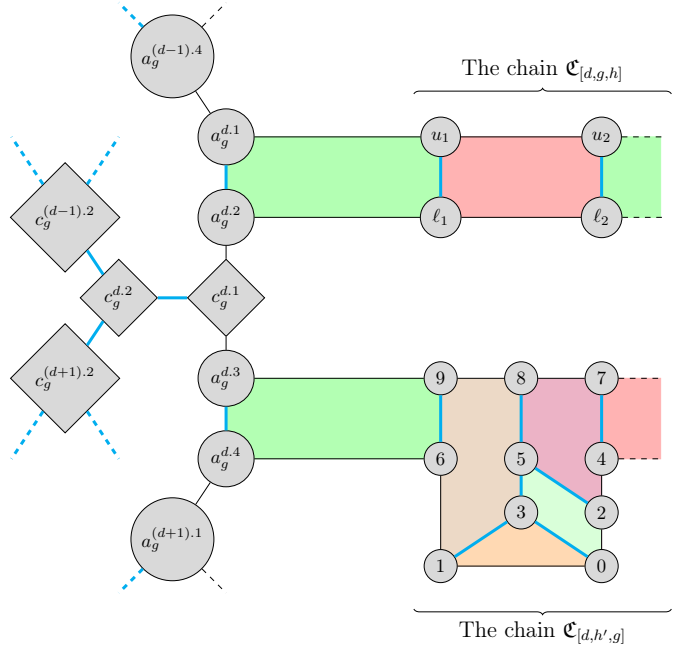

\subsection{The Roof Edge Cycles}\label{subsect:roofcycles}
In this section, we describe a set that covers all roof edges of the gadgets, which motivates its name. The following edges lie in exactly one cycle of $\mathcal{Z}_i$ and in exactly one cycle of $\mathcal{Z}_r$, the latter set being constructed in this subsection:
\begin{enumerate}
    \item the \textbf{upper chain edges}, that are the edges $\left\{\mathfrak{C}_{[d,g,h]}(u_j) , \mathfrak{C}_{[d,g,h]}(u_{j+1}) \right\}$,
    \item the \textbf{upper endgadget edges}, that are the edges $\left\{ \mathfrak{C}_{[d,g,h]}(4), \mathfrak{C}_{[d,g,h]}(2) \right\}$, $\left\{ \mathfrak{C}_{[d,g,h]}(2), \mathfrak{C}_{[d,g,h]}(0) \right\}$, $\left\{ \mathfrak{C}_{[d,g,h]}(0), \mathfrak{C}_{[d,g,h]}(1) \right\}$ and $\left\{ \mathfrak{C}_{[d,g,h]}(1), \mathfrak{C}_{[d,g,h]}(6) \right\}$,
    \item the \textbf{$s_d$-block-connecting edges}, that are the edges of the form $\left\{ {a_g}^{d.4}, {a_g}^{(d+1 \modd k).1} \right\}$, 
    \item the \textbf{upper connection edges}, that are the edges of the form $\left\{ a_g^{d.1} , \mathfrak{C}_{[d,g,h]}(u_1) \right\}$ and \\ $\left\{ \mathfrak{C}_{[d,g,h]}(6), a_h^{d.4} \right\}$.
\end{enumerate}

Each red edge in \cref{fig:cdcUpdated,fig:cyclesD10Updated} is covered by one cycle in $\mathcal{Z}_i$ and now also by a second cycle in $\mathcal{Z}_r$.
The light blue ones are the edges that are already covered by two cycles in the set $\mathcal{Z}_i$, and the thin black edges are the ones that still require a second cycle.

Firstly, we introduce some notation. For an edge $(g,h)\in \Cay_{G,S}$ we refer to the unique upper-level path from $a_g^{d.1}$ to $a_h^{d.4}$ that traverses only the upper chain edges, the upper endgadget edges in the chain $\mathfrak{C}_{[d,g,h]}$ and the upper connection edges as the \textbf{upper path}.

To compute a roof edge cycle for a given starting point we proceed as follows:
We start with an element $g \in G$ and the generator $s_1 \in S$. We use the upper path in the chain $\mathfrak{C}_{[1,g,h_1]}$ with $h_1 = gs_1$. Within $A_{h_1}$, we traverse the edge $\left\{a_{h_1}^{1.4}, a_{h_1}^{2.1}\right\}$ to enter its $s_2$-block. 
Then, we take the upper path from $A_{h_1}$ to $A_{h_2}$ with $h_2 = h_1 s_2$. We iterate the procedure of entering $A_h$ through a generator $s_d$ and immediately exiting it via the generator $s_{(d+1 \modd k)}$ along its upper path, continuing this process until a vertex is encountered for the second time. Whenever we enter $A_h$ via the last generator $s_k$, we save the element $h \in G$ in the set $H_{\textnormal{vis}}$ of visited elements of $G$. This process terminates, since we assume $G$ to be finite.
In summary, we compute one roof edge cycle, as well as the set $H_{\textnormal{vis}}$ of elements in $G$ whose blow-up graphs we entered via the generator $s_k$.

Now, we compute the set of all roof edge cycles.
We start with the full group $G$ and an \textbf{update set} $\tilde{G}$ initially set to $G$. In the set $\tilde{G}$, we keep track of the elements in $G$ which have yet to be visited via the generator $s_k$ in a roof edge cycle. First, we compute one roof edge cycle with the neutral element $1_G$ as the starting point, and we remove all those elements from $\tilde{G}$ that we have just visited via the $k$-th generator, i.e.\ all elements in the set $H_{\textnormal{vis}}$.
Iteratively, we compute another roof edge cycle by using a remaining element in the set $\tilde{G}$ as the starting point. We repeat this as long as $\tilde{G}$ is not empty. Afterwards, all roof edge cycles $\mathcal{Z}_r$ are computed.

To formally prove the correctness and completeness of the roof edge cycles, we ought to prove the following statements. We only state these results without proof, as their verification is straightforward. Additional insight into the correctness of our cycles will be provided in Section~\ref{secProperties}. 
The corresponding properties, adjusted to the edges specified in point 3, likewise apply to the set $\mathcal{Z}_b$ of base edge cycles.

\begin{remark}\label{lemmaCompleteness1}
    The following properties are satisfied for the roof edge cycles in $\mathcal{Z}_r$ computed as described above:
    \begin{enumerate}
        \item No two distinct cycles of $\mathcal{Z}_r$ intersect.
        \item Every roof edge cycle in $\mathcal{Z}_r$ intersects in at most one edge with any inner face cycle from the set $\mathcal{Z}_i$. 
        \item Every upper chain edge, every upper endgadget edge, every $s_d$-block-connecting edge, as well as every connection edge in the graph $D_{G,S}$ is contained in exactly one cycle of $\mathcal{Z}_r$.
        Moreover, the cycles in $\mathcal{Z}_r$ only contain precisely those edges.
    \end{enumerate}    
\end{remark}

\begin{figure}[ht]
    \centering
    \begin{subfigure}[b]{0.4\textwidth}
        \centering
        \resizebox{5cm}{!}{\begin{tikzpicture}[
    node distance = 3mm and 3mm,
         V/.style = {circle, draw, fill=gray!30},
    every edge quotes/.style = {auto, font=\footnotesize, sloped}
                        ]

\tikzstyle{overbrace text style}=[font=\large, above, pos=.5, yshift=4mm]
\tikzstyle{underbrace text style}=[font=\large, below, pos=.5, yshift=4mm]
\tikzstyle{overbrace style}=[decorate,decoration={brace,raise=3mm,amplitude=3pt}]
\tikzstyle{underbrace style}=[decorate,decoration={brace,raise=3mm,amplitude=3pt,mirror}]

\coordinate (1) at (3,1);
\coordinate (2) at (4.5, 1);
\coordinate (3) at (7.5,1);
\coordinate (4) at (9,1);

\begin{scope}[nodes=V]

\node[circle, draw] (7) at (3,1)   {$a_g^{1.1}$};
\node[circle, draw] (8) at (4.5,1)   {$a_g^{1.2}$};
\node[circle, draw] (11) at (7.5,1)   {$a_g^{1.3}$};
\node[circle, draw] (12) at (9,1)   {$a_g^{1.4}$};

\end{scope}

\draw

        (7) edge[dashed, ultra thick, red] (3, 2.7)
        (8) edge[dashed] (4.5, 2.7)
        (11) edge[dashed] (7.5, 2.7)
        (12) edge[dashed, ultra thick, red] (9, 2.7)

        (7) edge[ultra thick, cyan] (8)
        (8) edge[] (11)
        (11) edge[ultra thick, cyan] (12)
        (7) edge[bend right = 70, ultra thick, red] (12);

\draw [draw = none] (3, 1.5) -- (4.5, 1.5) node [overbrace text style, font=\small] {$\mathfrak{C}_{[1,g,h_1]}$};

\draw [draw = none] (7.5, 1.5) -- (9, 1.5) node [overbrace text style, font=\small] {$\mathfrak{C}_{[1,h_1^{\prime}, g]}$};

        \end{tikzpicture}}
        \caption{$k=1$}
        \label{fig:cyclesD1genUpdate}
    \end{subfigure}
    \hfill
    \begin{subfigure}[b]{0.4\textwidth}
        \centering
        \resizebox{5cm}{!}{
        \begin{tikzpicture}[
    node distance = 3mm and 3mm,
         V/.style = {circle, draw, fill=gray!30},
    every edge quotes/.style = {auto, font=\footnotesize, sloped}
                        ]

\tikzstyle{overbrace text style}=[font=\large, above, pos=.5, yshift=4mm]
\tikzstyle{underbrace text style}=[font=\large, below, pos=.5, yshift=4mm]
\tikzstyle{overbrace style}=[decorate,decoration={brace,raise=3mm,amplitude=3pt}]
\tikzstyle{underbrace style}=[decorate,decoration={brace,raise=3mm,amplitude=3pt,mirror}]

\begin{scope}[nodes=V]

\node[circle, draw] (19) at (9,-1)   {$a_g^{2.1}$};
\node[circle, draw] (20) at (7.5,-1)   {$a_g^{2.2}$};
\node[shape = diamond, draw] (21) at (6,-1)   {$c_g^{2.1}$};
\node[circle, draw] (23) at (4.5,-1)   {$a_g^{2.3}$};
\node[circle, draw] (24) at (3,-1)   {$a_g^{2.4}$};

\node[circle, draw] (7) at (3,1)   {$a_g^{1.1}$};
\node[circle, draw] (8) at (4.5,1)   {$a_g^{1.2}$};
\node[shape = diamond, draw] (9) at (6,1)   {$c_g^{1.1}$};
\node[circle, draw] (11) at (7.5,1)   {$a_g^{1.3}$};
\node[circle, draw] (12) at (9,1)   {$a_g^{1.4}$};

\end{scope}

\draw

        (7) edge[dashed, ultra thick, red] (3, 2.5)
        (8) edge[dashed] (4.5, 2.5)
        (11) edge[dashed] (7.5, 2.5)
        (12) edge[dashed, ultra thick, red] (9, 2.5)
        (9) edge[ultra thick, cyan] (21)
        (12) edge[ultra thick, red] (19)
        (24) edge[ultra thick, red] (7)

        (24) edge[dashed, ultra thick, red] (3, -2.5)
        (23) edge[dashed] (4.5, -2.5)
        (20) edge[dashed] (7.5, -2.5)
        (19) edge[dashed, ultra thick, red] (9, -2.5)

        (7) edge[ultra thick, cyan] (8)
        (8) edge[] (9)
        (9) edge[] (11)
        (11) edge[ultra thick, cyan] (12)

        (19) edge[ultra thick, cyan] (20)
        (20) edge[] (21)
        (21) edge[] (23)
        (23) edge[ultra thick, cyan] (24);

\draw [draw = none] (3, 1.5) -- (4.5, 1.5) node [overbrace text style, font=\small] {$\mathfrak{C}_{[1,g,h_1]}$};
\draw [draw = none] (3, -2.3) -- (4.5, -2.3) node [underbrace text style, font=\small] {$\mathfrak{C}_{[2,h_2^{\prime}, g]}$};

\draw [draw = none] (7.5, 1.5) -- (9, 1.5) node [overbrace text style, font=\small] {$\mathfrak{C}_{[1,h_1^{\prime}, g]}$};
\draw [draw = none] (7.5, -2.3) -- (9, -2.3) node [underbrace text style, font=\small] {$\mathfrak{C}_{[2,g,h_2]}$};

\end{tikzpicture}
        }
        \caption{$k=2$}
        \label{fig:agCycles2Gens}
    \end{subfigure}
    \begin{subfigure}[b]{0.99\textwidth}
        \centering
        \resizebox{12cm}{!}{
        \begin{tikzpicture}[
    node distance = 3mm and 3mm,
         V/.style = {circle, draw, fill=gray!30},
    every edge quotes/.style = {auto, font=\footnotesize, sloped}
                        ]

\tikzstyle{overbrace text style}=[font=\large, above, pos=.5, yshift=4mm]
\tikzstyle{underbrace text style}=[font=\large, below, pos=.5]
\tikzstyle{overbrace style}=[decorate,decoration={brace,raise=3mm,amplitude=3pt}]

\begin{scope}[nodes=V]

\node[circle, draw] (13) at (12,0)   {$a_g^{2.1}$};
\node[circle, draw] (14) at (12,-1.5)   {$a_g^{2.2}$};
\node[shape = diamond, draw] (15) at (12,-3)   {$c_g^{2.1}$};
\node[shape = diamond, draw] (16) at (10,-3)   {$c_g^{2.2}$};
\node[circle, draw] (17) at (12, -4.5)   {$a_g^{2.3}$};
\node[circle, draw] (18) at (12,-6)   {$a_g^{2.4}$};

\node[circle, draw] (19) at (9,-8)   {$a_g^{3.1}$};
\node[circle, draw] (20) at (7.5,-8)   {$a_g^{3.2}$};
\node[shape = diamond, draw] (21) at (6,-8)   {$c_g^{3.1}$};
\node[shape = diamond, draw] (22) at (6,-6)   {$c_g^{3.2}$};
\node[circle, draw] (23) at (4.5,-8)   {$a_g^{3.3}$};
\node[circle, draw] (24) at (3,-8)   {$a_g^{3.4}$};

\node[circle, draw] (1) at (0,0)   {$a_g^{d.4}$};
\node[circle, draw] (2) at (0,-1.5)   {$a_g^{4.3}$};
\node[shape = diamond, draw] (3) at (0,-3)   {$c_g^{4.1}$};
\node[shape = diamond, draw] (4) at (2,-3)   {$c_g^{4.2}$};
\node[circle, draw] (5) at (0,-4.5)   {$a_g^{4.2}$};
\node[circle, draw] (6) at (0,-6)   {$a_g^{4.1}$};

\node[circle, draw] (7) at (3,2)   {$a_g^{1.1}$};
\node[circle, draw] (8) at (4.5,2)   {$a_g^{1.2}$};
\node[shape = diamond, draw] (9) at (6,2)   {$c_g^{1.1}$};
\node[shape = diamond, draw] (10) at (6,0)   {$c_g^{1.2}$};
\node[circle, draw] (11) at (7.5,2)   {$a_g^{1.3}$};
\node[circle, draw] (12) at (9,2)   {$a_g^{1.4}$};

\end{scope}

\draw

        (7) edge[dashed, ultra thick, red] (3, 3.5)
        (8) edge[dashed] (4.5, 3.5)
        (11) edge[dashed] (7.5, 3.5)
        (12) edge[dashed, ultra thick,red] (9, 3.5)

        (24) edge[dashed, ultra thick, red] (3, -9.5)
        (23) edge[dashed] (4.5, -9.5)
        (20) edge[dashed] (7.5, -9.5)
        (19) edge[dashed, ultra thick, red] (9, -9.5)

        (1) edge[dashed, ultra thick, red] (-1.5, 0)
        (2) edge[dashed] (-1.5, -1.5)
        (5) edge[dashed] (-1.5, -4.5)
        (6) edge[dashed, ultra thick, red] (-1.5, -6) 

        (13) edge[dashed, ultra thick, red] (13.5, 0)
        (14) edge[dashed] (13.5, -1.5)
        (17) edge[dashed] (13.5, -4.5)
        (18) edge[dashed, ultra thick, red] (13.5, -6) 

        (1) edge[ultra thick, cyan] (2)
        (2) edge[] (3)
        (3) edge[ultra thick, cyan] (4)
        (3) edge[] (5)
        (5) edge[ultra thick, cyan] (6)

        (4) edge[ultra thick, cyan] (10)
        (10) edge[ultra thick, cyan] (16)
        (16) edge[ultra thick, cyan] (22)
        (22) edge[ultra thick, cyan] (4)

        (1) edge[ultra thick, red] (7)
        (12) edge[ultra thick, red] (13)
        (18) edge[ultra thick, red] (19)
        (24) edge[ultra thick, red] (6)

        (7) edge[ultra thick, cyan] (8)
        (8) edge[] (9)
        (9) edge[ultra thick, cyan] (10)
        (9) edge[] (11)
        (11) edge[ultra thick, cyan] (12)

        (13) edge[ultra thick, cyan] (14)
        (14) edge[] (15)
        (15) edge[ultra thick, cyan] (16)
        (15) edge[] (17)
        (17) edge[ultra thick, cyan] (18)

        (19) edge[ultra thick, cyan] (20)
        (20) edge[] (21)
        (21) edge[ultra thick, cyan] (22)
        (21) edge[] (23)
        (23) edge[ultra thick, cyan] (24);

        \draw [draw = none] (3, 2.5) -- (4.5, 2.5) node [overbrace text style, font=\small] {$\mathfrak{C}_{[1,g,h_1]}$};
        \draw [draw = none] (3, -9) -- (4.5, -9) node [underbrace text style, font=\small] {$\mathfrak{C}_{[3,h_3^{\prime}, g]}$};
        
        \draw [draw = none] (7.5, 2.5) -- (9, 2.5) node [overbrace text style, font=\small] {$\mathfrak{C}_{[1,h_1^{\prime}, g]}$};
        \draw [draw = none] (7.5, -9) -- (9, -9) node [underbrace text style, font=\small] {$\mathfrak{C}_{[3,g,h_3]}$};
        
        \draw [draw = none] (-1.5, -6) -- (-0.5, -6) node [overbrace text style, font=\small] {$\mathfrak{C}_{[4,g,h_4]}$};
        \draw [draw = none] (13.5, -6) -- (12.5, -6) node [overbrace text style, font=\small] {$\mathfrak{C}_{[2,h_2^{\prime}, g]}$};
        
        \draw [draw = none] (-1.5, -1.5) -- (-0.5, -1.5) node [overbrace text style, font=\small] {$\mathfrak{C}_{[4,h_4^{\prime}, g]}$};
        \draw [draw = none] (13.5, -1.5) -- (12.5, -1.5) node [overbrace text style, font=\small] {$\mathfrak{C}_{[2,g,h_2]}$};

        \end{tikzpicture}
        }
        \caption{$k=4$}
    \end{subfigure}
\caption{\centering The roof edge cycles (red) in the blow-up graphs $A_g$ for various $k$.}
\label{fig:cdcUpdated}
\end{figure}
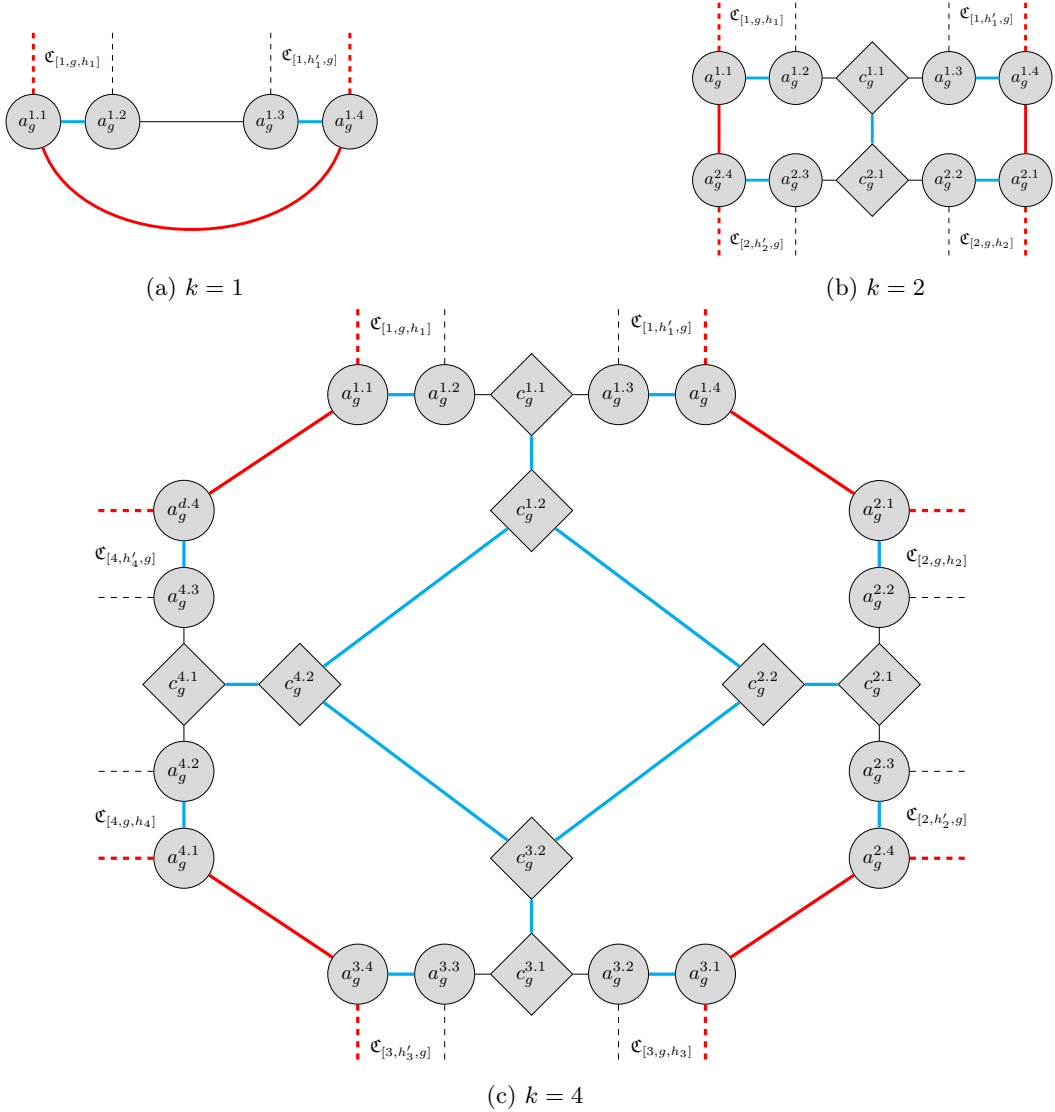

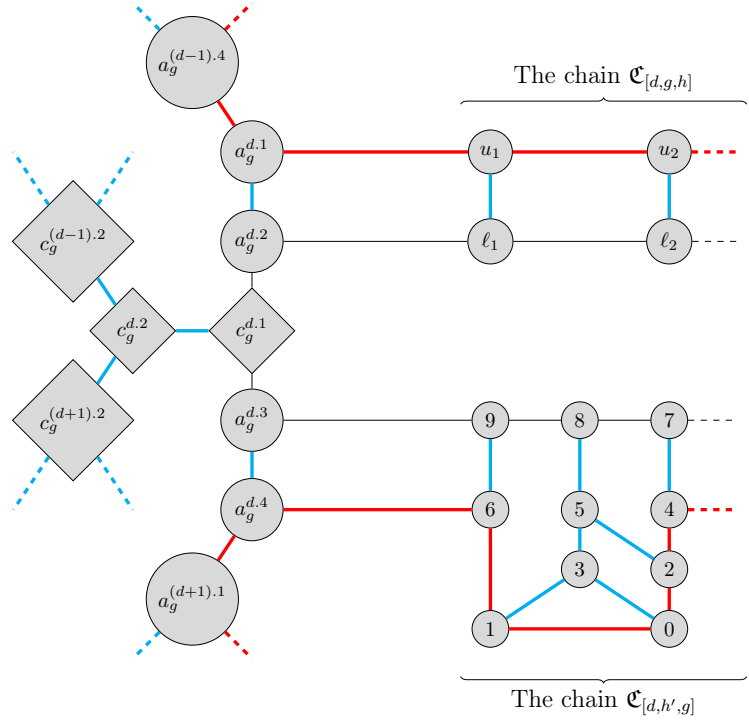
\begin{figure}[ht!]
        \centering
        \resizebox{10cm}{!}{
        \begin{tikzpicture}[
    node distance = 3mm and 3mm,
         V/.style = {circle, draw, fill=gray!30},
    every edge quotes/.style = {auto, font=\footnotesize, sloped}
                        ]

\tikzstyle{overbrace text style}=[font=\large, above, pos=.5, yshift=4mm]
\tikzstyle{underbrace text style}=[font=\large, below, pos=.5, yshift=-4mm]
\tikzstyle{overbrace style}=[decorate,decoration={brace,raise=3mm,amplitude=3pt}]
\tikzstyle{underbrace style}=[decorate,decoration={brace,raise=3mm,amplitude=3pt,mirror}]

\begin{scope}[nodes=V]

\node (a) at (4, 0)    {$u_1$};
\node (b) at (4, -1.5)          {$\ell_1$};

\node (c) at (7, 0) {$u_2$};
\node (d) at (7, -1.5) {$\ell_2$};


\node (a8) at (4, -4.5) {9};
\node (a5) at (4, -6) {6};
\node (a7) at (5.5, -4.5) {8};
\node (a6) at (7, -4.5) {7};
\node (a4) at (7, -6) {4};
\node (a3) at (5.5, -7) {3};
\node (a2) at (4, -8) {1};
\node (a1) at (7, -8) {$0$};
\node (a0) at (7, -7) {$2$};
\node (a9) at (5.5, -6) {$5$};

\node[circle, draw] (1) at (0,0)   {$a_g^{d.1}$};
\node[circle, draw] (2) at (0,-1.5)   {$a_g^{d.2}$};
\node[shape = diamond] (3) at (0,-3)   {$c_g^{d.1}$};
\node[shape = diamond] (4) at (-2,-3)   {$c_g^{d.2}$};
\node[circle, draw] (5) at (0,-4.5)   {$a_g^{d.3}$};
\node[circle] (6) at (0,-6)   {$a_g^{d.4}$};

\node(27) at (-1, 1.5) {$a_g^{(d-1).4}$};
\node (28) at (-1, -7.5) {$a_g^{(d+1).1}$};

\node[shape = diamond] (25) at (-3, -1.5) {$c_g^{(d-1).2}$};
\node[shape = diamond] (26) at (-3, -4.5) {$c_g^{(d+1).2}$};

\end{scope}

\draw   

        (1) edge[ultra thick, red] (a)
        (2) edge (b)

        (a) edge[ultra thick, red] (c)
        (b) edge (d)
        (c) edge[ultra thick, cyan] (d)
        (a) edge[ultra thick, cyan] (b)

        (c) edge[dashed, ultra thick, red] (8.2,0)
        (d) edge[dashed] (8.2,-1.5)

        (a4) edge[dashed, ultra thick, red] (8.1, -6)
        (a6) edge[dashed] (8.1, -4.5)

        (5) edge (a8)
        (6) edge[ultra thick, red] (a5)
        
        (27) edge[dashed, ultra thick, cyan] (-2, 2.5)
        (27) edge[dashed, ultra thick, red] (0, 2.5)

        (28) edge[dashed, ultra thick, cyan] (-2, -8.5)
        (28) edge[dashed, ultra thick, red] (0, -8.5)

        (25) edge[dashed, ultra thick, cyan] (-4,0)
        (25) edge[dashed, ultra thick, cyan] (-2, 0)

        (26) edge[dashed, ultra thick, cyan] (-4, -6)
        (26) edge[dashed, ultra thick, cyan] (-2, -6)

        (1) edge [ultra thick, red] (27)
        (6) edge[ultra thick, red] (28)

        (4) edge[ultra thick, cyan] (25)
        (4) edge[ultra thick, cyan] (26)

        (1) edge[ultra thick, cyan] (2)
        (2) edge (3)
        (3) edge[ultra thick, cyan] (4)
        (3) edge (5)
        (5) edge[ultra thick, cyan] (6)

        (a1) edge[ultra thick, red] (a2)
        (a1) edge[ultra thick, cyan] (a3)
        (a1) edge[ultra thick, red] (a0)
        (a0) edge[ultra thick, red] (a4)
        (a2) edge[ultra thick, red] (a5)
        (a2) edge[ultra thick, cyan] (a3)
        (a3) edge[ultra thick, cyan] (a9)
        (a0) edge[ultra thick, cyan] (a9)
        (a9) edge[ultra thick, cyan] (a7)
        (a4) edge[ultra thick, cyan] (a6)
        (a5) edge[ultra thick, cyan] (a8)
        (a7) edge (a6)
        (a7) edge (a8);

\draw [overbrace style] (3.5, 0.5) -- (8.3, 0.5) node [overbrace text style] {The chain $\mathfrak{C}_{[d,g,h]}$};

\draw [underbrace style] (3.5, -8.5) -- (8.3, -8.5) node [underbrace text style] {The chain $\mathfrak{C}_{[d,h^{\prime},g]}$};

\end{tikzpicture}
        }
        \caption{The roof edge cycles (red) in a $d$-chain.}
        \label{fig:cyclesD10Updated}
\end{figure}

\subsection{The Base Edge Cycles}\label{subsect:basecycles}
In this section, we present a set of cycles that covers all base edges, hence the name. The following edges lie in precisely one cycle of $\mathcal{Z}_i$ and in precisely one cycle of $\mathcal{Z}_b$, the latter being constructed in this subsection:
\begin{enumerate}
    \item the \textbf{lower chain edges}, that are the edges $\left\{ \mathfrak{C}_{[d,g,h]}(\ell_j), \mathfrak{C}_{[d,g,h]}( \ell_{j+1})\right\}$,
    \item the \textbf{lower endgadget edges}, that are the edges $\left\{ \mathfrak{C}_{[d,g,h]}(6), \mathfrak{C}_{[d,g,h]}(7) \right\}$ and \\$\left\{ \mathfrak{C}_{[d,g,h]}(7), \mathfrak{C}_{[d,g,h]}(8) \right\}$,
    \item the \textbf{remaining $s_d$-block edges}, more precisely the edges $\left\{ a_g^{d.2}, c_g^{d.1}\right\}$ and $\left\{ c_g^{d.1}, a_g^{d.3} \right\}$.
    \item the \textbf{base connection edges}, that are the edges of the form $\left\{ a_g^{d.2} , \mathfrak{C}_{[d,g,h]}(\ell_1) \right\}$ and \\ $\left\{ \mathfrak{C}_{[d,g,h]}(9), a_h^{d.3} \right\}$.
\end{enumerate}
Each green edge in \cref{fig:cdcUpdated2,fig:cyclesD10Updated1} is covered by one cycle in $\mathcal{Z}_i$ and now also by a second cycle in $\mathcal{Z}_b$.
The light blue and red edges are the ones that were already covered by two cycles in our previous sets $\mathcal{Z}_i$ and $\mathcal{Z}_r$. 

The case that $G$ is generated by one element ($k=1$) is exceptional, as the graph $D_{G,S}$ does not contain any center vertices. 
Hence, we simply omit traversing the center vertices $c_g^{d.1}$ and only traverse the two connector vertices $a_g^{d.2}$ and $a_g^{d.3}$.
Similarly to the upper path, for an edge $(g,h) \in E(\Cay_{G,S})$ of colour $d$, we define the unique \textbf{lower path} from $a_g^{d.2}$ to $a_h^{d.3}$ that only traverses the lower chain, the lower endgadget edges of the $d$-chain $\mathfrak{C}_{[d,g,h]}$ and the base connection edges.

To compute a base edge cycle for a given starting point we proceed as follows:  We start with an element $g \in G$ and a fixed generator $s_d \in S$. Firstly, we trace along the lower path of the chain $\mathfrak{C}_{[d,g,h_1]}$ with $h_1 = gs_d$.
Having arrived in $A_{h_1}$, we go to the center vertex $c_{h_1}^{d.1}$ of the $s_d$-block in $A_{h_1}$, then to the second connector vertex of the $s_d$-block, $a_{h_1}^{d.2}$, and use the lower path in $\mathfrak{C}_{[d,h_1, h_2]}$ with $h_2 = h_1s_d$. 
Moreover, we add every group element $h$ that we have visited on the path to the set $\bar{H}_{\textnormal{vis}}$.
We repeat this process until we visit a vertex for the second time. 
In summary, we compute one base edge cycle and the set $\bar{H}_{\textnormal{vis}}$ of group elements we visited.

Now, we compute the set of all base edge cycles similarly to before, the only difference being that we define a set $\tilde{G}_d$ for each index of generators $1 \leq d \leq k$ individually. 
For each $d \in [k]$, we proceed similarly as for the roof edge cycles: We take a remaining element $g$ in $\tilde{G}_d$, compute the base edge cycle starting in $g$ and the set $\bar{H}_{\textnormal{vis}}$, remove the elements in $\bar{H}_{\textnormal{vis}}$ from $\tilde{G}_d$, and add the base edge cycle we just computed to $\mathcal{Z}_b$.
The result of this iteration for all $d \in [k]$ is the set $\mathcal{Z}_b$ of all base edge cycles.

\cref{fig:cdcUpdated2,fig:cyclesD10Updated1} illustrate the final cycle double cover of $D_{G,S}$ according to the three sets $\mathcal{Z}_i, \mathcal{Z}_r$ and $\mathcal{Z}_b$ constructed above. As before, the blue edges are those covered twice by face cycles of $\mathcal{Z}_i$. The red edges are covered twice as well, once through $\mathcal{Z}_i$ and once through $\mathcal{Z}_r$. Finally, the green edges are covered once by $\mathcal{Z}_i$ and once by $\mathcal{Z}_b$.

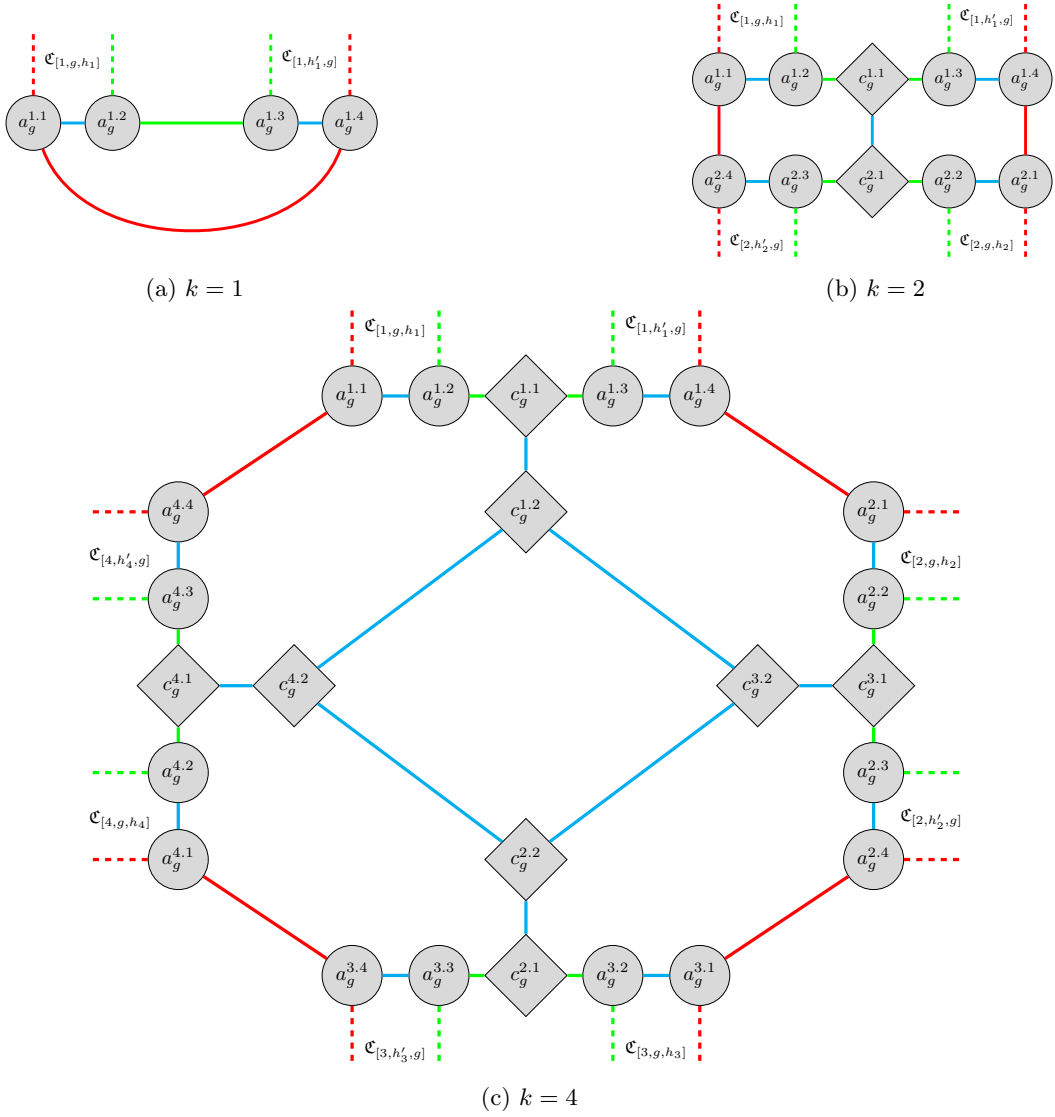
\begin{figure}[ht]
    \begin{subfigure}[b]{0.4\textwidth}
        \centering
        \resizebox{5cm}{!}{\begin{tikzpicture}[
    node distance = 3mm and 3mm,
         V/.style = {circle, draw, fill=gray!30},
    every edge quotes/.style = {auto, font=\footnotesize, sloped}
                        ]

\tikzstyle{overbrace text style}=[font=\large, above, pos=.5, yshift=4mm]
\tikzstyle{underbrace text style}=[font=\large, below, pos=.5, yshift=4mm]
\tikzstyle{overbrace style}=[decorate,decoration={brace,raise=3mm,amplitude=3pt}]
\tikzstyle{underbrace style}=[decorate,decoration={brace,raise=3mm,amplitude=3pt,mirror}]

\coordinate (1) at (3,1);
\coordinate (2) at (4.5, 1);
\coordinate (3) at (7.5,1);
\coordinate (4) at (9,1);

\begin{scope}[nodes=V]

\node[circle, draw] (7) at (3,1)   {$a_g^{1.1}$};
\node[circle, draw] (8) at (4.5,1)   {$a_g^{1.2}$};
\node[circle, draw] (11) at (7.5,1)   {$a_g^{1.3}$};
\node[circle, draw] (12) at (9,1)   {$a_g^{1.4}$};

\end{scope}

\draw

        (7) edge[dashed, ultra thick, red] (3, 2.7)
        (8) edge[dashed, ultra thick, green] (4.5, 2.7)
        (11) edge[dashed, ultra thick, green] (7.5, 2.7)
        (12) edge[dashed, ultra thick, red] (9, 2.7)

        (7) edge[ultra thick, cyan] (8)
        (8) edge[ultra thick, green] (11)
        (11) edge[ultra thick, cyan] (12)
        (7) edge[bend right = 70, ultra thick, red] (12);

\draw [draw = none] (3, 1.5) -- (4.5, 1.5) node [overbrace text style, font=\small] {$\mathfrak{C}_{[1,g,h_1]}$};

\draw [draw = none] (7.5, 1.5) -- (9, 1.5) node [overbrace text style, font=\small] {$\mathfrak{C}_{[1,h_1^{\prime}, g]}$};

        \end{tikzpicture}}
        \caption{$k=1$}
        \label{fig:cyclesD1genUpdate2}
    \end{subfigure}
    \hfill
    \begin{subfigure}[b]{0.4\textwidth}
        \centering
        \resizebox{5cm}{!}{
        \begin{tikzpicture}[
    node distance = 3mm and 3mm,
         V/.style = {circle, draw, fill=gray!30},
    every edge quotes/.style = {auto, font=\footnotesize, sloped}
                        ]

\tikzstyle{overbrace text style}=[font=\large, above, pos=.5, yshift=4mm]
\tikzstyle{underbrace text style}=[font=\large, below, pos=.5, yshift=4mm]
\tikzstyle{overbrace style}=[decorate,decoration={brace,raise=3mm,amplitude=3pt}]
\tikzstyle{underbrace style}=[decorate,decoration={brace,raise=3mm,amplitude=3pt,mirror}]

\begin{scope}[nodes=V]

\node[circle, draw] (19) at (9,-1)   {$a_g^{2.1}$};
\node[circle, draw] (20) at (7.5,-1)   {$a_g^{2.2}$};
\node[shape = diamond, draw] (21) at (6,-1)   {$c_g^{2.1}$};
\node[circle, draw] (23) at (4.5,-1)   {$a_g^{2.3}$};
\node[circle, draw] (24) at (3,-1)   {$a_g^{2.4}$};

\node[circle, draw] (7) at (3,1)   {$a_g^{1.1}$};
\node[circle, draw] (8) at (4.5,1)   {$a_g^{1.2}$};
\node[shape = diamond, draw] (9) at (6,1)   {$c_g^{1.1}$};
\node[circle, draw] (11) at (7.5,1)   {$a_g^{1.3}$};
\node[circle, draw] (12) at (9,1)   {$a_g^{1.4}$};

\end{scope}

\draw

        (7) edge[dashed, ultra thick, red] (3, 2.5)
        (8) edge[dashed, ultra thick, green] (4.5, 2.5)
        (11) edge[dashed, ultra thick, green] (7.5, 2.5)
        (12) edge[dashed, ultra thick, red] (9, 2.5)
        (9) edge[ultra thick, cyan] (21)
        (12) edge[ultra thick, red] (19)
        (24) edge[ultra thick, red] (7)

        (24) edge[dashed, ultra thick, red] (3, -2.5)
        (23) edge[dashed, ultra thick, green] (4.5, -2.5)
        (20) edge[dashed, ultra thick, green] (7.5, -2.5)
        (19) edge[dashed, ultra thick, red] (9, -2.5)

        (7) edge[ultra thick, cyan] (8)
        (8) edge[ultra thick, green] (9)
        (9) edge[ultra thick, green] (11)
        (11) edge[ultra thick, cyan] (12)

        (19) edge[ultra thick, cyan] (20)
        (20) edge[ultra thick, green] (21)
        (21) edge[ultra thick, green] (23)
        (23) edge[ultra thick, cyan] (24);

\draw [draw = none] (3, 1.5) -- (4.5, 1.5) node [overbrace text style, font=\small] {$\mathfrak{C}_{[1,g,h_1]}$};
\draw [draw = none] (3, -2.3) -- (4.5, -2.3) node [underbrace text style, font=\small] {$\mathfrak{C}_{[2,h_2^{\prime}, g]}$};

\draw [draw = none] (7.5, 1.5) -- (9, 1.5) node [overbrace text style, font=\small] {$\mathfrak{C}_{[1,h_1^{\prime}, g]}$};
\draw [draw = none] (7.5, -2.3) -- (9, -2.3) node [underbrace text style, font=\small] {$\mathfrak{C}_{[2,g,h_2]}$};

\end{tikzpicture}
        }
        \caption{$k=2$}
        \label{fig:agCycles2Gens2}
    \end{subfigure}
    \begin{subfigure}[b]{0.99\textwidth}
        \centering
        \resizebox{12cm}{!}{
        \begin{tikzpicture}[
    node distance = 3mm and 3mm,
         V/.style = {circle, draw, fill=gray!30},
    every edge quotes/.style = {auto, font=\footnotesize, sloped}
                        ]

\tikzstyle{overbrace text style}=[font=\large, above, pos=.5, yshift=4mm]
\tikzstyle{underbrace text style}=[font=\large, below, pos=.5]
\tikzstyle{overbrace style}=[decorate,decoration={brace,raise=3mm,amplitude=3pt}]

\begin{scope}[nodes=V]

\node[circle, draw] (13) at (12,0)   {$a_g^{2.1}$};
\node[circle, draw] (14) at (12,-1.5)   {$a_g^{2.2}$};
\node[shape = diamond, draw] (15) at (12,-3)   {$c_g^{3.1}$};
\node[shape = diamond, draw] (16) at (10,-3)   {$c_g^{3.2}$};
\node[circle, draw] (17) at (12, -4.5)   {$a_g^{2.3}$};
\node[circle, draw] (18) at (12,-6)   {$a_g^{2.4}$};

\node[circle, draw] (19) at (9,-8)   {$a_g^{3.1}$};
\node[circle, draw] (20) at (7.5,-8)   {$a_g^{3.2}$};
\node[shape = diamond, draw] (21) at (6,-8)   {$c_g^{2.1}$};
\node[shape = diamond, draw] (22) at (6,-6)   {$c_g^{2.2}$};
\node[circle, draw] (23) at (4.5,-8)   {$a_g^{3.3}$};
\node[circle, draw] (24) at (3,-8)   {$a_g^{3.4}$};

\node[circle, draw] (1) at (0,0)   {$a_g^{4.4}$};
\node[circle, draw] (2) at (0,-1.5)   {$a_g^{4.3}$};
\node[shape = diamond, draw] (3) at (0,-3)   {$c_g^{4.1}$};
\node[shape = diamond, draw] (4) at (2,-3)   {$c_g^{4.2}$};
\node[circle, draw] (5) at (0,-4.5)   {$a_g^{4.2}$};
\node[circle, draw] (6) at (0,-6)   {$a_g^{4.1}$};

\node[circle, draw] (7) at (3,2)   {$a_g^{1.1}$};
\node[circle, draw] (8) at (4.5,2)   {$a_g^{1.2}$};
\node[shape = diamond, draw] (9) at (6,2)   {$c_g^{1.1}$};
\node[shape = diamond, draw] (10) at (6,0)   {$c_g^{1.2}$};
\node[circle, draw] (11) at (7.5,2)   {$a_g^{1.3}$};
\node[circle, draw] (12) at (9,2)   {$a_g^{1.4}$};

\end{scope}

\draw

        (7) edge[dashed, ultra thick, red] (3, 3.5)
        (8) edge[dashed, ultra thick, green] (4.5, 3.5)
        (11) edge[dashed, ultra thick, green] (7.5, 3.5)
        (12) edge[dashed, ultra thick, red] (9, 3.5)

        (24) edge[dashed, ultra thick, red] (3, -9.5)
        (23) edge[dashed, ultra thick, green] (4.5, -9.5)
        (20) edge[dashed, ultra thick, green] (7.5, -9.5)
        (19) edge[dashed, ultra thick, red] (9, -9.5)

        (1) edge[dashed, ultra thick, red] (-1.5, 0)
        (2) edge[dashed, ultra thick, green] (-1.5, -1.5)
        (5) edge[dashed, ultra thick, green] (-1.5, -4.5)
        (6) edge[dashed, ultra thick, red] (-1.5, -6) 

        (13) edge[dashed, ultra thick, red] (13.5, 0)
        (14) edge[dashed, ultra thick, green] (13.5, -1.5)
        (17) edge[dashed, ultra thick, green] (13.5, -4.5)
        (18) edge[dashed, ultra thick, red] (13.5, -6) 

        (1) edge[ultra thick, cyan] (2)
        (2) edge[ultra thick, green] (3)
        (3) edge[ultra thick, cyan] (4)
        (3) edge[ultra thick, green] (5)
        (5) edge[ultra thick, cyan] (6)

        (4) edge[ultra thick, cyan] (10)
        (10) edge[ultra thick, cyan] (16)
        (16) edge[ultra thick, cyan] (22)
        (22) edge[ultra thick, cyan] (4)

        (1) edge[ultra thick, red] (7)
        (12) edge[ultra thick, red] (13)
        (18) edge[ultra thick, red] (19)
        (24) edge[ultra thick, red] (6)

        (7) edge[ultra thick, cyan] (8)
        (8) edge[ultra thick, green] (9)
        (9) edge[ultra thick, cyan] (10)
        (9) edge[ultra thick, green] (11)
        (11) edge[ultra thick, cyan] (12)

        (13) edge[ultra thick, cyan] (14)
        (14) edge[ultra thick, green] (15)
        (15) edge[ultra thick, cyan] (16)
        (15) edge[ultra thick, green] (17)
        (17) edge[ultra thick, cyan] (18)

        (19) edge[ultra thick, cyan] (20)
        (20) edge[ultra thick, green] (21)
        (21) edge[ultra thick, cyan] (22)
        (21) edge[ultra thick, green] (23)
        (23) edge[ultra thick, cyan] (24);

        \draw [draw = none] (3, 2.5) -- (4.5, 2.5) node [overbrace text style, font=\small] {$\mathfrak{C}_{[1,g,h_1]}$};
        \draw [draw = none] (3, -9) -- (4.5, -9) node [underbrace text style, font=\small] {$\mathfrak{C}_{[3,h_3^{\prime}, g]}$};
        
        \draw [draw = none] (7.5, 2.5) -- (9, 2.5) node [overbrace text style, font=\small] {$\mathfrak{C}_{[1,h_1^{\prime}, g]}$};
        \draw [draw = none] (7.5, -9) -- (9, -9) node [underbrace text style, font=\small] {$\mathfrak{C}_{[3,g,h_3]}$};
        
        \draw [draw = none] (-1.5, -6) -- (-0.5, -6) node [overbrace text style, font=\small] {$\mathfrak{C}_{[4,g,h_4]}$};
        \draw [draw = none] (13.5, -6) -- (12.5, -6) node [overbrace text style, font=\small] {$\mathfrak{C}_{[2,h_2^{\prime}, g]}$};
        
        \draw [draw = none] (-1.5, -1.5) -- (-0.5, -1.5) node [overbrace text style, font=\small] {$\mathfrak{C}_{[4,h_4^{\prime}, g]}$};
        \draw [draw = none] (13.5, -1.5) -- (12.5, -1.5) node [overbrace text style, font=\small] {$\mathfrak{C}_{[2,g,h_2]}$};

        \end{tikzpicture}
        }
        \caption{$k=4$}
    \end{subfigure}
    \caption{\centering The base edge cycles (green) in the blow-up graphs $A_g$ for various $k$.}
    \label{fig:cdcUpdated2}
\end{figure}
    
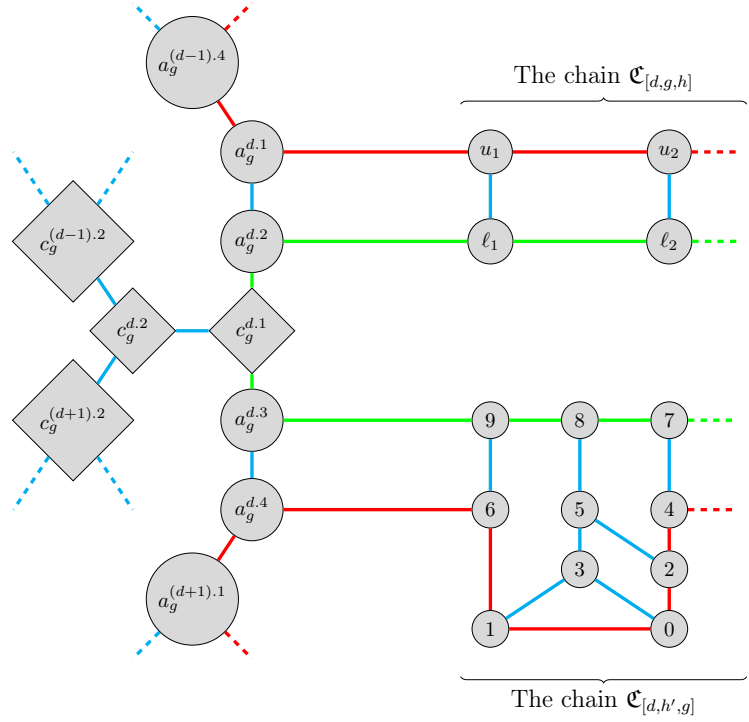
\begin{figure}[ht!]
        \centering
        \resizebox{10cm}{!}{
        \begin{tikzpicture}[
    node distance = 3mm and 3mm,
         V/.style = {circle, draw, fill=gray!30},
    every edge quotes/.style = {auto, font=\footnotesize, sloped}
                        ]

\tikzstyle{overbrace text style}=[font=\large, above, pos=.5, yshift=4mm]
\tikzstyle{underbrace text style}=[font=\large, below, pos=.5, yshift=-4mm]
\tikzstyle{overbrace style}=[decorate,decoration={brace,raise=3mm,amplitude=3pt}]
\tikzstyle{underbrace style}=[decorate,decoration={brace,raise=3mm,amplitude=3pt,mirror}]

\begin{scope}[nodes=V]

\node (a) at (4, 0)    {$u_1$};
\node (b) at (4, -1.5)          {$\ell_1$};

\node (c) at (7, 0) {$u_2$};
\node (d) at (7, -1.5) {$\ell_2$};


\node (a8) at (4, -4.5) {9};
\node (a5) at (4, -6) {6};
\node (a7) at (5.5, -4.5) {8};
\node (a6) at (7, -4.5) {7};
\node (a4) at (7, -6) {4};
\node (a3) at (5.5, -7) {3};
\node (a2) at (4, -8) {1};
\node (a1) at (7, -8) {$0$};
\node (a0) at (7, -7) {$2$};
\node (a9) at (5.5, -6) {$5$};

\node[circle, draw] (1) at (0,0)   {$a_g^{d.1}$};
\node[circle, draw] (2) at (0,-1.5)   {$a_g^{d.2}$};
\node[shape = diamond] (3) at (0,-3)   {$c_g^{d.1}$};
\node[shape = diamond] (4) at (-2,-3)   {$c_g^{d.2}$};
\node[circle, draw] (5) at (0,-4.5)   {$a_g^{d.3}$};
\node[circle] (6) at (0,-6)   {$a_g^{d.4}$};

\node(27) at (-1, 1.5) {$a_g^{(d-1).4}$};
\node (28) at (-1, -7.5) {$a_g^{(d+1).1}$};

\node[shape = diamond] (25) at (-3, -1.5) {$c_g^{(d-1).2}$};
\node[shape = diamond] (26) at (-3, -4.5) {$c_g^{(d+1).2}$};

\end{scope}

\draw   

        (1) edge[ultra thick, red] (a)
        (2) edge[ultra thick, green] (b)

        (a) edge[ultra thick, red] (c)
        (b) edge[ultra thick, green] (d)
        (c) edge[ultra thick, cyan] (d)
        (a) edge[ultra thick, cyan] (b)

        (c) edge[dashed, ultra thick, red] (8.2,0)
        (d) edge[dashed, ultra thick, green] (8.2,-1.5)

        (a4) edge[dashed, ultra thick, red] (8.1, -6)
        (a6) edge[dashed, ultra thick, green] (8.1, -4.5)

        (5) edge[ultra thick, green] (a8)
        (6) edge[ultra thick, red] (a5)
        
        (27) edge[dashed, ultra thick, cyan] (-2, 2.5)
        (27) edge[dashed, ultra thick, red] (0, 2.5)

        (28) edge[dashed, ultra thick, cyan] (-2, -8.5)
        (28) edge[dashed, ultra thick, red] (0, -8.5)

        (25) edge[dashed, ultra thick, cyan] (-4,0)
        (25) edge[dashed, ultra thick, cyan] (-2, 0)

        (26) edge[dashed, ultra thick, cyan] (-4, -6)
        (26) edge[dashed, ultra thick, cyan] (-2, -6)

        (1) edge [ultra thick, red] (27)
        (6) edge[ultra thick, red] (28)

        (4) edge[ultra thick, cyan] (25)
        (4) edge[ultra thick, cyan] (26)

        (1) edge[ultra thick, cyan] (2)
        (2) edge[ultra thick, green] (3)
        (3) edge[ultra thick, cyan] (4)
        (3) edge[ultra thick, green] (5)
        (5) edge[ultra thick, cyan] (6)

        (a0) edge[ultra thick, cyan] (a9)

        (a1) edge[ultra thick, red] (a2)
        (a1) edge[ultra thick, cyan] (a3)
        (a1) edge[ultra thick, red] (a0)
        (a0) edge[ultra thick, red] (a4)
        (a2) edge[ultra thick, red] (a5)
        (a2) edge[ultra thick, cyan] (a3)
        (a3) edge[ultra thick, cyan] (a9)
        (a9) edge[ultra thick, cyan] (a7)
        (a4) edge[ultra thick, cyan] (a6)
        (a5) edge[ultra thick, cyan] (a8)
        (a7) edge[ultra thick, green] (a6)
        (a7) edge[ultra thick, green] (a8);

\draw [overbrace style] (3.5, 0.5) -- (8.3, 0.5) node [overbrace text style] {The chain $\mathfrak{C}_{[d,g,h]}$};

\draw [underbrace style] (3.5, -8.5) -- (8.3, -8.5) node [underbrace text style] {The chain $\mathfrak{C}_{[d,h^{\prime},g]}$};

\end{tikzpicture}
        }
        \caption{The base edge cycles (green) in a $d$-chain.}
        \label{fig:cyclesD10Updated1}
\end{figure}

\section{Properties of the CDC}\label{secProperties}
In this section, we prove that the cycle double cover (CDC) $\mathcal{Z}=\mathcal{Z}_i\cup\mathcal{Z}_r\cup\mathcal{Z}_b$ of $D_{G,S}$ constructed in the previous section induces a polyhedral map. We also examine several combinatorial properties of this CDC. 
The main result is that the constructed CDC, and therefore the induced polyhedral map, is invariant under the action of the automorphism group of $D_{G,S}$.
Again, we consider a group $G$ with a generating set $S = \{s_1, \ldots, s_k\}$ and a fixed ordering on its elements, as well as the associated graph $D_{G,S}$.
Firstly, we consider the connection between the roof and base edge cycles and the powers of elements in $G$.
\begin{lemma}\label{lemmaCyclesPowers}
    Let $s := \prod_{d=1}^k s_d$ be the product of all generators $s_d \in S$ according to our ordering and $g\in G$ the starting point.
    \begin{enumerate}
        \item The set $H_{\textnormal{vis}}$, defined in \Cref{subsect:roofcycles}, computed together with the roof edge cycles is exactly the left coset $g\langle s \rangle$.
        \item The set $\bar{H}_{\textnormal{vis}}$, defined in \Cref{subsect:basecycles}, computed together with the base edge cycles is exactly the left coset $g \langle s_d \rangle$ for every $d \in [k]$.
    \end{enumerate}
\end{lemma}
\begin{proof}
    The argument for both types of cycles is analogous, so we only sketch the proof for the first set $H_{\textnormal{vis}}$ computed by the roof edge cycles. 
    Let $g \in G$. 
    Recall the procedure for computing the roof edge cycles: When contracting our graph construction $D_{G,S}$ back to the Cayley graph $\Cay_{G,S}$, the path we trace is induced by the right multiplication with every generator $s_1, \ldots, s_k$ in increasing order, restarting with $s_1$ after $s_k$, starting in the element $g$. 
    Also, we save every element we have visited via the last generator $s_k$. 
    Thus, we save precisely those elements $h \in G$ in $H_{\textnormal{vis}}$ such that there is some $p \in \mathbb{N}$ with $g(s_1 \cdot \; \cdots \; \cdot s_k)^p = h$, or, in other words, $h \in g\langle s \rangle$. 
    When lifting this path to our graph construction, we observe that the same property still holds. Hence,  $H_{\textnormal{vis}}=g\langle s\rangle$.
    The argument for the base edge cycles is the same by only considering the multiplication from the right with one fixed generator $s_d \in S$.
\end{proof}

The above lemma implies already that the CDC induces a polyhedral map.
\begin{theorem}\label{theoremZ1CDC}
    Let $G$ be a group with a generating set $S = \{ s_1, \ldots, s_k\}$. Let $\mathcal{Z}$ be defined as $\mathcal{Z} = \mathcal{Z}_i \cup \mathcal{Z}_r \cup \mathcal{Z}_b$, where $\mathcal{Z}_i$ is the set of inner face cycles, $\mathcal{Z}_r$ the set of roof edge cycles, and $\mathcal{Z}_b$ the set of base edge cycles. Then $\mathcal{Z}$ is a CDC that induces a polyhedral map of $D_{G,S}$.
\end{theorem}
\begin{proof}
    We have to show three properties of the computed paths: The fact that they are cycles, as well as the completeness and correctness of $\mathcal{Z}$.

    Firstly, it is easy to see that the roof edge cycles and the base edge cycles are indeed cycles. This follows immediately from (1) in \cref{lemmaCyclesPowers}. 
    Seeing that we start in some $a_g^{1.1}$ in a roof edge cycle, or in some $a_g^{d.2}$ for $d \in [k]$ in a base edge cycle, we finish the computation of a cycle as soon as we revisit $a_g^{1.1}$, or $a_g^{d.2}$ respectively, having visited no other vertex twice before.

    We can infer the completeness and correctness of the cycles using basic group theoretical arguments. The powers of a single element $g \in G$ form a subgroup in $G$, that is $\langle g \rangle \leq G$. Now, we know that $G$ is the setwise disjoint union of all left cosets of $\langle g \rangle$ in $G$. 

    Again, we set $s := \prod_{d=1}^k s_d$.
    From \cref{lemmaCyclesPowers}, we know that the roof and base edge cycles trace precisely along the elements in the left cosets of $g\langle s \rangle$ and $g \langle s_d \rangle, d \in [k]$, respectively. 
    Thus, no two roof or base edge cycles intersect with one another, as they traverse the $d$-chains and blow-up graphs of different left cosets induced by $s$ in $G$. 
    Also, no roof edge cycle intersects with a base edge cycle, simply because they traverse different chain edges, as well as different blow-up graph vertices.
    Finally, since $G$ is the disjoint union of its left cosets, and since we end the computation of the roof and base edge cycles when $\tilde{G}$ is empty, we can infer the completeness: Every edge in a roof or a base edge cycle occurs in precisely one cycle along a left coset of the form $g \langle f \rangle$ for $f \in G$.

    Finally, it is easy to see that any edge $e \in E(D_{G,S})$ is either contained in two cycles of $\mathcal{Z}_i$, or in only one cycle of $\mathcal{Z}_i$ and in a second cycle of either $\mathcal{Z}_r$ or $\mathcal{Z}_b$. So, $\mathcal{Z}$ is a CDC of $D_{G,S}$ satisfying the properties for inducing a polyhedral map.
\end{proof}

In fact, \cref{lemmaCyclesPowers} provides a way of obtaining various polyhedral maps at once. The inner face cycles only depend on $G$ and the cardinality of $S$. The base edge cycles only depend on the generating set $S = \left\{ s_1, \ldots, s_k\right\}$ as a set, but the roof edge cycles depend very much on the ordering we chose for the generators. Thus, permuting the generators may, in general, yield a different element $s = \prod_{d=1}^k s_d$, and therefore a different set of roof edge cycles.
Now, we determine the cardinality of the constructed CDC.

\begin{lemma}\label{lemmaSizeZ}
    Let $G$ be a group of order $n$ with a generating set $S$ of cardinality $k$. Let $\mathcal{Z}$ denote the CDC of the graph $D_{G,S}$ from \cref{theoremZ1CDC}, and set $s := \prod_{d=1}^k s_d$. Then, $\mathcal{Z}$ contains
    \begin{enumerate}
        \item $n\left( \frac{k^2+15k}{2} + 1 + \frac{1}{|s|} + \sum_{d=1}^k \frac{1}{|s_d|}\right)$ cycles if $k \geq 3$,
        \item $n \left(17 + \frac{1}{|s|} + \frac{1}{|s_1|} + \frac{1}{|s_2|} \right)$ cycles if $k = 2$, and
        \item $8n+2$ cycles if $k=1$.
    \end{enumerate}
\end{lemma}
\begin{proof}
    We first consider the case $k \geq 3$.
    Every blow-up graph $A_g$ accounts for $k+1$ many face cycles, given by the faces delimited by two adjacent $s_d$-blocks, and the face cycle induced by the cycle over all center vertices of superscript two. 
    Every $d$-chain $\mathfrak{C}_{[d,g,h]}$ contains precisely $d+6$ cycles when counting the face cycles as well as the face cycles delimited by a $d$-chain and a blow-up graph.
    Summing over all elements $g \in G$, this totals
    \[  n \left( (k+1) + \sum_{d=1}^k (d+6)\right) = n\left( \frac{k^2+15k}{2} + 1 \right) \]
    face cycles in $D_{G,S}$.
    Moreover, we have the roof and base edge cycles. Using \cref{lemmaCyclesPowers}, the element $s = \prod_{d=1}^k s_d$ accounts for $[G : \langle s \rangle ]$ many cycles, and every generator $s_d \in S$ for $[ G : \langle s_d \rangle ]$ many cycles. Adding these sizes to the above formula yields the identity for the case $k \geq 3$.

    For $k=2$, we have one less face cycle within every blow-up graph, as the interior center vertex cycles vanish in this case. Other than that, all cycles from the case $k \geq 3$ also exist in $D_{G,S}$ when $k = 2$. Therefore, subtracting $n$ from the above identity yields the number of cycles for this case. 

    Finally, for the case $k=1$, we only have one face cycle per group element $g \in G$. Moreover, we only have one roof and base edge cycle, since the group is cyclic. Finally, since we only have one generator, we have seven face cycles in total on each $d$-chain. The sum yields the above formula.
\end{proof}

The automorphisms of $D_{G,S}$ are induced by a left-regular action of $G$ on itself. 
Moreover, in \cref{lemmaDChainsPermus,lemmaDNGPermus}, we have seen that both the $d$-chains and the blow-up graphs can only be mapped as a whole to other $d$-chains and blow-up graphs in accordance with the underlying left-regular action. In other words, any automorphism of the graph $D_{G,S}$ does not change the local structure of the graph, only the broader one which determines which blow-up graph corresponds to which element. Formally, we can prove the following result on the invariance of the cycles of the constructed CDC under the action of the group of the graph $D_{G,S}$.
\begin{theorem}\label{lemmaInvarianceCDC}
    For a group $G=\langle S\rangle = \langle s_1, \ldots, s_k \rangle$, let $\mathcal{Z}$ be the constructed CDC inducing a polyhedral map of the graph $D_{G,S}$ from \cref{theoremZ1CDC}. Then, $\prescript{G}{}{\mathcal{Z}} = \prescript{\Aut(D_{G,S})}{}{\mathcal{Z}} = \mathcal{Z}$, i.e.\ $\mathcal{Z}$ is invariant under the action of the group of the graph $D_{G,S}$.
\end{theorem}
In fact, we can even go one step further and prove that $\prescript{G}{}{\mathcal{Z}_i} = \mathcal{Z}_i$ for $i \in [3]$, i.e.\ face cycles are mapped to face cycles, roof edge cycles to roof edge cycles and base edge cycles to base edge cycles.

Given the size of the graph $D_{G,S}$ and \cref{lemmaSizeZ}, we can compute the Euler characteristic of the polyhedral map induced by the constructed CDC.

\begin{theorem}\label{theoremEulerChar}
    Let $G$ be a group of order $n$ with a generating set $S$ of cardinality $k$. Let $\mathcal{Z}$ be the CDC of $D_{G,S}$ from \cref{theoremZ1CDC}, and define $s := \prod_{d=1}^ks_d$. Then, the Euler characteristic of the polyhedral map of $D_{G,S}$ induced by $\mathcal{Z}$ is given by
    \begin{enumerate}
        \item $n \left(1 - k + \frac{1}{|s|} + \sum_{d=1}^k \frac{1}{|s_d|} \right)$ if $k \geq 3$,
        \item $n \left(- 1 + \frac{1}{|s|} + \frac{1}{|s_1|} + \frac{1}{|s_2|} \right) $ if $k=2$, and
        \item $2$ if $k=1$.
    \end{enumerate}
\end{theorem}
Hence, for $k=1$ the CDC of the cubic graph $D_{G,S}$ gives rise to a spherical map.

\section*{Acknowledgements}
R.\ Akpanya, M.\ Weiß and A.\ C.\ Niemeyer gratefully acknowledge the funding by the Deutsche Forschungsgemeinschaft (DFG, German Research Foundation) in the framework of the Collaborative Research Centre CRC/TRR 280 “Design Strategies for Material-Minimized Carbon Reinforced Concrete Structures – Principles of a New Approach to Construction” (project ID 417002380).
Furthermore, R.\ Akpanya was supported by a grant from the Simons Foundation (SFI-MPS-Infrastructure-00008650, JV).

\bibliographystyle{plain}
\bibliography{main}

\end{document}